\documentclass[11pt]{amsart}
\usepackage[T1]{fontenc}
\usepackage{lmodern}
\usepackage[a4paper,textwidth=6.25in,textheight=9.25in,centering]{geometry}
\usepackage{amsmath,amssymb,mathtools,microtype}
\usepackage[hidelinks]{hyperref}
\newtheorem{theorem}{Theorem}[section]
\newtheorem{lemma}[theorem]{Lemma}
\newtheorem{proposition}[theorem]{Proposition}
\newtheorem{corollary}[theorem]{Corollary}
\theoremstyle{remark}
\newtheorem{remark}[theorem]{Remark}
\numberwithin{equation}{section}
\newcommand{\R}{\mathbb R}
\newcommand{\HH}{\mathbb H^2}
\newcommand{\dd}{\,d}
\newcommand{\supp}{\operatorname{supp}}
\newcommand{\Vol}{\operatorname{Vol}}
\allowdisplaybreaks[1]
\title[Cubic moments on surfaces]{ Cubic moments of eigenfunctions on surfaces}
\date{}
\author{Xing Wang and Cheng Zhang}
\address{School of Mathematics, Hunan University, Changsha, HN 410012, China}
\email{xingwang@hnu.edu.cn}
\address{Yau Mathematical Sciences Center,
	Tsinghua University,
	Beijing, BJ 100084, China}
\email{czhang98@tsinghua.edu.cn}

\subjclass[2010]{Primary 35P20, 58J50; Secondary 42B20}

\keywords{Laplace eigenfunctions, cubic moments, Gaussian curvature, hyperbolic surfaces, Gaussian moments, oscillatory integrals}
\hypersetup{pdftitle={Local and global cubic moments of eigenfunctions on surfaces},
	pdfkeywords={Laplace eigenfunctions, cubic moments, Gaussian curvature, hyperbolic surfaces, Gaussian moments, oscillatory integrals},
	pdfsubject={Sharp local cubic bounds and global cubic decay on complete hyperbolic surfaces}}
\begin{document}
	\begin{abstract}
		We identify Gaussian curvature as a local mechanism for cubic cancellation
		of Laplace eigenfunctions on Riemannian surfaces, without arithmetic or
		dynamical assumptions. For $L^2$-normalized local solutions at frequency
		$\lambda$, cubic integrals against smooth compactly supported weights
		tend to zero whenever the curvature is nonzero almost everywhere where
		the weight is nonzero. Finite-order nondegeneracy yields sharp polynomial
		decay rates. In particular, on closed surfaces with nowhere-vanishing
		curvature, we obtain the rate $O(\lambda^{-1/6})$, attained by spherical
		harmonics on the round sphere. This proves the cubic case of the Gaussian
		moment conjecture for all normalized real eigenfunctions on closed
		negatively curved surfaces. The local proof rests on an exact identity
		relating the reduced phase at a Fermat point to Gaussian curvature.
		On complete hyperbolic surfaces with positive systole, we also prove
		global $O((\log\lambda)^{-3/2})$ decay for $L^2$-normalized eigenfunctions
		and smooth weights with bounded covariant derivatives through order two,
		without assuming finite area.
	\end{abstract}
	\maketitle
	\enlargethispage{3pt}
	
	\section{Introduction}\label{sec:introduction}
	
	Berry's random-wave model \cite{Berry} predicts Gaussian value
	statistics for high-energy eigenfunctions in chaotic systems; see also
	\cite{AurichSteiner,McDonaldKaufman,Nonnenmacher,Srednicki}. A corresponding Gaussian moment conjecture predicts that, on a fixed
	closed connected negatively curved surface $(M,g)$, every sequence of
	real eigenfunctions satisfying
	$-\Delta_g e_\lambda=\lambda^2e_\lambda$ and
	$\|e_\lambda\|_{L^2(M)}=1$ obeys, for each fixed
	$w\in C^\infty(M)$ and integer $q\ge0$,
	\begin{equation}\label{eq:gaussian-prediction}
		\int_M w e_\lambda^q\dd V_g
		\longrightarrow
		\frac{\mathbb E[Z^q]}{\Vol_g(M)^{q/2}}
		\int_M w\dd V_g,
		\qquad Z\sim N(0,1),
	\end{equation}
	as $\lambda\to\infty$. Related formulations include the Gaussian
	value-distribution question of Hejhal and Rackner
	\cite[Section~5]{HejhalRackner}, the unweighted spherical analogue
	of B\"ocherer, Sarnak, and Schulze-Pillot \cite[Conjecture~2]{BSS},
	Humphries' compact-set Gaussian moment conjecture \cite[Conjecture~1.1]{Humphries},
	and smoothly weighted moments on the modular surface by Huang
	\cite[Section~1.1]{Huang}.
	
For $q=1$, integration by parts gives
$\int_M w e_\lambda\dd V_g=O_{w,N}(\lambda^{-N})$ for every $N>0$.
For $q=2$, \eqref{eq:gaussian-prediction} predicts spatial equidistribution,
a consequence of quantum unique ergodicity (QUE) as formulated by
Rudnick and Sarnak \cite{RudnickSarnak}.
The quantum ergodicity theorem of Shnirelman, Colin de Verdi\`ere,
and Zelditch \cite{Shnirelman,ColinDeVerdiere,Zelditch1987}
establishes this convergence along a density-one subsequence.
Lindenstrauss \cite{Lindenstrauss} proved QUE for Hecke eigenfunctions
on compact arithmetic congruence surfaces.
Soundararajan \cite{Soundararajan} ruled out escape of mass, thereby completing Lindenstrauss's QUE result for Hecke--Maass cusp forms on the modular surface.
The full-sequence prediction remains open for general closed hyperbolic surfaces.
See also Dyatlov and Jin \cite{DyatlovJin} and Dyatlov, Jin,
and Nonnenmacher \cite{DyatlovJinNonnenmacher} for full-support
results on semiclassical measures.
For $q=3$, the first nonlinear odd moment,
\eqref{eq:gaussian-prediction} predicts
$\int_M w e_\lambda^3\dd V_g\to0$.
This signed cancellation follows neither from QUE nor from
$L^3$-norm bounds and is the focus of this paper.
	
	On the modular surface, Watson \cite[Theorem~5]{Watson} proved unweighted cubic-moment decay for Hecke--Maass cusp forms, and Huang \cite[Theorem~1.1]{Huang} established a weighted version. B\"ocherer, Sarnak, and Schulze-Pillot \cite[Proposition~2.4]{BSS} proved cubic decay for Hecke spherical harmonics on the round sphere, whose geodesic flow is completely integrable. By contrast, resonant eigenfunctions on a flat torus can have nondecaying cubic moments (see Proposition~\ref{prop:sharp-flat}). These examples show that chaotic dynamics is not necessary for cubic decay, while complete integrability alone does not guarantee it. Although arithmetic structure plays an important role in the cited results, this contrast raises a natural question: what geometric conditions ensure cubic-moment decay independently of arithmetic or dynamical assumptions?
	
	Our results show that Gaussian curvature provides a local mechanism for
	cubic cancellation. If $K_g$ is nonzero almost everywhere on
	$\{w\ne0\}$, the weighted cubic moment tends to zero. When the order
	of vanishing of $K_g$ is uniformly bounded, we obtain polynomial decay
	rates with exponents that are sharp for the corresponding order. These
	conclusions hold for local solutions of the eigenfunction equation,
	without arithmetic or dynamical assumptions.
	
	\subsection{Local cubic decay and sharpness}
	Let $(M,g)$ be a smooth Riemannian surface without boundary, with
	Gaussian curvature $K_g$, Riemannian distance $d_g$, and area measure
	$dV_g$. We use the nonpositive Laplace--Beltrami operator $\Delta_g$
	and the Levi--Civita covariant derivative $\nabla$.
	
	\begin{theorem}\label{thm:upper}
		Let $U\Subset M$ be open, $w\in C_c^\infty(U)$, and
		$e_\lambda\in C^\infty(U)\cap L^2(U)$ satisfy, for $\lambda\ge1$,
		\[
		-\Delta_g e_\lambda=\lambda^2e_\lambda\quad\text{on }U,
		\qquad \|e_\lambda\|_{L^2(U)}=1.
		\]
		The following bounds hold.
		
		\textup{(i)} Without a curvature assumption,
		\begin{equation}\label{eq:uniform-bound}
			\left|\int_Uw e_\lambda^3\dd V_g\right|
			\le C_{g,w,U}.
		\end{equation}
		
		\textup{(ii)} If $K_g\ne0$ almost everywhere on $\{w\ne0\}$, then as $\lambda\to\infty$
		\begin{equation}\label{eq:upper}
			\left|\int_Uw e_\lambda^3\dd V_g\right|=o(1).
		\end{equation}
		
		\textup{(iii)}  If $\sum_{j=0}^m|\nabla^jK_g(x)|_g\ge\kappa$ on $\supp w$,
		where $m\ge0$ is an integer and $\kappa>0$, then
		\begin{equation}\label{eq:finite-order-corollary}
			\left|\int_Uw e_\lambda^3\dd V_g\right|
			\le C_{g,w,U,m,\kappa}\lambda^{-\frac1{6+2m}}.
		\end{equation}
		The same bounds hold with $e_\lambda^3$ replaced by $u_1u_2u_3$,
		where each $u_j\in C^\infty(U)\cap L^2(U)$ satisfies
		$-\Delta_gu_j=\lambda^2u_j$ on $U$ and
		$\|u_j\|_{L^2(U)}=1$. If $M$ is closed, we may take $U=M$ and
		$w\in C^\infty(M)$.
		
	\end{theorem}
	
	In particular, nowhere-vanishing curvature gives the decay rate
	$O(\lambda^{-1/6})$. More generally, the rate depends on the order to
	which the curvature is allowed to vanish. These estimates involve only
	the geometry and $L^2(U)$-normalized solutions on $U$, with no
	assumptions beyond $U$. In particular, the estimates apply to locally $L^2$-normalized Eisenstein series on hyperbolic surfaces.
	
	The local results are sharp in both qualitative and quantitative respects.
	Without a curvature assumption, cubic cancellation can fail: a fixed
	flat torus admits normalized eigenfunctions with
	$\int_M e_{\lambda_n}^3\dd V_g=\sqrt{2/3}$, and the Euclidean plane
	provides a local counterpart (Proposition~\ref{prop:sharp-flat}).
	Almost-everywhere nonvanishing of $K_g$ on $\{w\ne0\}$ ensures
	convergence, but gives no decay rate uniform over this class of metrics.
	Indeed, for any prescribed rate tending to zero, one can choose a fixed
	complete smooth metric with $K_g<0$ wherever $w\ne0$, together with
	normalized local solutions whose cubic moments tend to zero but exceed
	that rate by an unbounded factor along a frequency sequence
	(Proposition~\ref{prop:sharp-slow}). Under the finite-order condition,
	the exponent $1/(6+2m)$ is optimal for every $m\ge0$: it is attained by
	normalized local solutions on a fixed complete surface for all
	sufficiently large $\lambda$, while a sequence of spherical harmonics
	on the round sphere attains the exponent $1/6$
	(Proposition~\ref{prop:sharp-finite}).

		The finite-order condition in Theorem~\ref{thm:upper}\textup{(iii)}
		can also be expressed by an integral lower bound: for some $c,r_*>0$,
		\[
		\int_{B_g(x,r)}|K_g|\dd V_g\ge cr^{m+2}
		\qquad(x\in\supp w,\ 0<r<r_*),
		\]
		where $B_g(x,r)=\{y:d_g(x,y)<r\}\subset U$.
		If $(M,g)$ is connected and real analytic with $K_g\not\equiv0$,
		then some finite $m$ meets the condition on every fixed $\supp w$.
		Consequently,
		for some $\gamma>0$,
		\begin{equation}\label{eq:analytic-decay}
			\left|\int_Uw e_\lambda^3\dd V_g\right|
			\le C_{g,w,U}\lambda^{-\gamma}.
		\end{equation}

	\begin{corollary}\label{cor:constant-curvature}
		Let $(M,g)$ be a closed Riemannian surface with everywhere positive
		or everywhere negative Gaussian curvature. If
		$e_\lambda\in C^\infty(M)$ satisfies
		$-\Delta_g e_\lambda=\lambda^2e_\lambda$ on $M$, $\lambda\ge1$,
		and $\|e_\lambda\|_{L^2(M)}=1$, then
		\begin{equation}
			\left|\int_M e_\lambda^3\dd V_g\right|\le C_{M,g}\lambda^{-\frac16}.
		\end{equation}
		In particular, this holds for the standard round sphere $S^2$ and
		all closed hyperbolic surfaces.
	\end{corollary}
	
The same bound holds for any fixed smooth weight. In particular, it establishes the $q=3$ case of the Gaussian moment conjecture \eqref{eq:gaussian-prediction} for every $L^2$-normalized real eigenfunction on any closed negatively curved surface, although the $q=2$ case (QUE) remains open in general. For Hecke spherical harmonics on the round sphere, Böcherer, Sarnak, and Schulze-Pillot \cite[Proposition~2.4]{BSS} proved that their unweighted cubic moments tend to zero. Our result extends this to all spherical harmonics and arbitrary fixed smooth weights, with the sharp decay rate $O(\lambda^{-1/6})$.
	
	\subsection{Global cubic moments on hyperbolic surfaces}
	For closed surfaces, the local theorem already controls the full cubic
	moment. On noncompact surfaces, however, it does not directly control
	integrals against weights extending to infinity; cusps, when present,
	require separate analysis. We therefore consider complete hyperbolic
	surfaces with $K_g=-1$. Write
	$\operatorname{sys}(M)$ for the infimum of the lengths of nonconstant
	closed geodesics (with $\inf\varnothing=\infty$), and, for a smooth
	weight $w$, let
	\[
	\|w\|_{C_b^2(M)}=\sum_{j=0}^2\sup_M|\nabla^jw|_g.
	\]
	
	\begin{theorem}\label{thm:global}
		For every $s_0>0$ there exists $C_{s_0}>0$ such that the
		following holds. Let $(M,g)$ be a complete hyperbolic surface
		without boundary with $\operatorname{sys}(M)\ge s_0$.
		Every $L^2$-normalized eigenfunction
		$e_\lambda\in C^\infty(M)\cap L^2(M)$ satisfying
		$-\Delta_g e_\lambda=\lambda^2e_\lambda$, $\lambda\ge2$,
		belongs to $L^3(M)$. Moreover, for every
		$w\in C^\infty(M)$ with $\|w\|_{C_b^2(M)}<\infty$,
		\begin{equation}\label{eq:global}
			\left|\int_M w e_\lambda^3\dd V_g\right|
			\le C_{s_0}\|w\|_{C_b^2(M)}
			(\log\lambda)^{-3/2}.
		\end{equation}
		The same bound holds with $e_\lambda^3$ replaced by
		$u_1u_2u_3$, where $u_1,u_2,u_3$ are any three such
		$L^2$-normalized eigenfunctions with eigenvalue $\lambda^2$.
	\end{theorem}
	
	The constant in Theorem~\ref{thm:global} depends on the systole lower
	bound, but not on the area or the number of cusps. Moreover, the
	weight need not have compact support or approach a limit in a cusp.
	Taking $w=1$ gives cubic-moment decay for $L^2$-normalized eigenfunctions
	on every complete finite-area
	hyperbolic surface. The same conclusion holds on complete infinite-area
	hyperbolic surfaces with positive systole, whenever such $L^2$
	eigenfunctions exist. 
	
	\begin{remark}\label{rem:cusp-free}
		By Patterson's theorem \cite{Patterson1975}, geometrically finite hyperbolic surfaces of infinite area have no $L^2$ eigenvalues in $[1/4,\infty)$. In particular, noncompact geometrically finite cusp-free surfaces admit no high-frequency $L^2$ eigenfunctions. For geometrically infinite cusp-free surfaces, the existence of such eigenfunctions remains open in general \cite[Introduction]{WeichWolf}. Nevertheless, if a complete cusp-free surface with $\operatorname{sys}(M)\ge s_0$ admits an $L^2$-normalized eigenfunction $e_\lambda$ as in Theorem~\ref{thm:global}, then \eqref{eq:global-shallow}, with $w_H=1$, yields the stronger estimate
		\[
		\left|\int_M w e_\lambda^3\dd V_g\right|
		\le C_{s_0}\|w\|_{C_b^2(M)}\lambda^{-1/6}.
		\]
		The same bound holds with $e_\lambda^3$ replaced by $u_1u_2u_3$, where the $u_j$ are any three such $L^2$-normalized eigenfunctions with eigenvalue $\lambda^2$.
	\end{remark}
	
	\subsection{Relation to previous work}
Sogge's $L^3$ estimate  gives an $O(\lambda^{1/4})$ bound for the cubic moments of $L^2$-normalized eigenfunctions on closed surfaces \cite{Sogge1988}. In contrast, we exploits cancellation to obtain uniform bounds and curvature-dependent decay for cubic moments. Arithmetic estimates for cubic moments rely on triple-product identities and automorphic $L$-functions. Watson's formula \cite[Theorem~3]{Watson} relates triple products of Hecke eigenforms to central $L$-values and yields the unweighted bound $O_\varepsilon(\lambda^{-1/6+\varepsilon})$ on the modular surface \cite[Theorem~5]{Watson}. For fixed smooth compactly supported weights, Huang \cite[Theorem~1.1]{Huang} obtained $O_{w,\varepsilon}(\lambda^{-1/12+\varepsilon})$ using spectral expansions and moments of $L$-functions. Guo proved the exponent $-1/6+\varepsilon$ for mixed cubic products \cite[Theorem~1.12]{Guo} and the weighted bound $O_{w,\varepsilon}(t^{-1/3+\varepsilon})$ for Eisenstein series $E(\cdot,\frac12+it)$ \cite[Theorem~1.2]{Guo}. Huang and Li \cite{HuangLiCubic} studied spectral averages of squared cubic moments of Hecke--Maass cusp forms and Eisenstein series. In contrast, our local estimates require no Hecke assumptions and allow vanishing or sign-changing curvature. When the curvature is nowhere zero on the support of the weight, they yield the sharp rate $\lambda^{-1/6}$ without an $\varepsilon$ loss. Our global estimate also requires neither finite area nor compact support of the weight.

On the round sphere, B\"ocherer, Sarnak, and Schulze-Pillot \cite[Proposition~2.4]{BSS} established unweighted cubic decay for Hecke spherical harmonics, while Marinucci and Wigman \cite{MarinucciWigman} proved a central limit theorem for cubic moments of Gaussian random spherical harmonics. Sarnak \cite{SarnakProducts} and Bernstein and Reznikov \cite{BernsteinReznikov1999,BernsteinReznikov2010} studied triple products using representation-theoretic methods. Bonifacio \cite{BonifacioHyperbolic} obtained spectral bounds on closed hyperbolic manifolds, while Wyman \cite{Wyman} related triple products to frequency triangles and spectral sums.

Related work on eigenfunction statistics includes local mass equidistribution of Eisenstein series by Luo and Sarnak \cite{LuoSarnak1995} and its microlocal analogue by Jakobson \cite{JakobsonEisenstein}. Quantum ergodicity in cusped settings was studied by Zelditch \cite{Zelditch1991} and Bonthonneau and Zelditch \cite{BonthonneauZelditch}, while quantum variance was investigated by Luo and Sarnak \cite{LuoSarnakVariance}, Sarnak and Zhao \cite{SarnakZhao}, Huang \cite{HuangVariance}, and Huang and Lester \cite{HuangLester}. Fourth-moment asymptotics were obtained by Buttcane and Khan \cite{ButtcaneKhan} under the generalized Lindel\"of hypothesis, Humphries and Khan \cite{HumphriesKhan} unconditionally for a dihedral family, and Djankovi\'c and Khan for regularized Eisenstein series \cite{DjankovicKhan} and averaged truncated Eisenstein series \cite{DjankovicKhanTruncated}. Ki \cite{Ki} proved subpolynomial $L^4$ bounds for Hecke--Maass cusp forms. For background on hyperbolic spectra and QUE, see Sarnak \cite{Sarnak2003,Sarnak2011} and Dyatlov \cite{Dyatlov2023}.
	
	\subsection{Proof outline and organization}
	Section~\ref{sec:local} establishes local $L^2$ and localization estimates
	from Sogge's oscillatory representation \cite{Sogge}.
	The geometric core of Section~\ref{sec:upper-proofs} is the relation
	between Fermat points and Gaussian curvature. The relevant spatial
	critical point is a local Fermat point, where the sum of three distances
	has a local minimum and the joining geodesics meet at angles of $2\pi/3$.
	Stationary phase at this Fermat point gives a trilinear oscillatory
	integral. The identity \eqref{eq:curvature-identity} expresses the
	derivative controlling nondegeneracy of the reduced phase as a
	nonvanishing multiple of Gaussian curvature at the Fermat point.
	A $TT^*$ argument and the two-sided fold estimate of Pan and Sogge
	\cite[Theorem~1.1]{PanSogge} then yield the $\lambda^{-1/6}$ rate
	when the curvature is nonzero. Localization and induction on the order
	of vanishing give the sharp finite-order rates, while a decomposition
	by curvature size gives qualitative decay under almost-everywhere
	nonvanishing. Section~\ref{sec:global} combines the local estimates
	with Fourier--Bessel analysis in the cusps to obtain global cubic decay.
	Section~\ref{sec:sharpness} constructs fixed surfaces and normalized
	solutions showing that the local results are sharp.
	
\subsection{Notation}	Constants $C,c>0$ and implicit constants may change from line to line
	and depend on fixed data and indicated parameters, but not on
	$\lambda$. For nonnegative $a,b$, we write $a\approx b$ when
	$cb\le a\le Cb$ with such constants.
	
	\subsection*{Acknowledgments.}
	The authors would like to thank Christopher Sogge and Bingrong Huang
	for helpful discussions and comments.
	C.Z. is partially supported by National Key R\&D Program of China
	No.~2024YFA1015300 and NSFC Grant No.~12371097.
	X.W. is partially supported by NSFC Grant No.~12671119.
	\section{Local integral estimates}\label{sec:local}
	
	Sogge's local oscillatory representation is the starting point for
	the $L^2$ and localization estimates developed here for
	Theorem~\ref{thm:upper}.
	
	We work with three $L^2(U)$-normalized solutions at the same frequency:
	\[
	-\Delta_gu_j=\lambda^2u_j\quad\text{on }U,\qquad
	\|u_j\|_{L^2(U)}=1,\qquad j=1,2,3.
	\]
	After a fixed finite partition of unity, we may work in normal
	coordinates such that
	\[
	\supp w\subset B(0,r_0),\qquad B(0,10r_0)\Subset U,\qquad
	\tfrac12 I\le g\le2I.
	\]
	Throughout Sections~\ref{sec:local}--\ref{sec:upper-proofs},
	$B(x,r)=\{y\in\R^2:|y-x|<r\}$ denotes a Euclidean coordinate ball.
	Choose $r_0>0$ sufficiently small that any two points of
	$B(0,10r_0)$ are joined by a unique minimizing geodesic contained
	in that ball. Unless a scale is specified, supports and coordinate neighborhoods
	are fixed. Uniform smooth bounds mean $C^N$ bounds for every fixed
	$N$, independent of $\lambda$ and the small parameters; $O_{C^N}(r)$
	denotes a remainder bounded by $C_Nr$ in $C^N$. Unless stated otherwise,
	surface integrals use $dV_g$ and parameter integrals use Lebesgue measure.
	The determinant $\text{det}_g$ is computed in an oriented
	$g$-orthonormal frame.
	
	\subsection{Sogge's representation}\label{sec:sogge}
	Choose $\chi_0\in C_c^\infty(B(0,4r_0))$ equal to one near
	$B(0,2r_0)$, and extend the metric near $\supp\chi_0$ to a closed
	surface. Let $P=\sqrt{-\Delta_g}$ and choose
	\[
	\varrho\in\mathcal S(\R),\qquad \varrho(0)=1,\qquad
	\supp\widehat\varrho\subset(\tau_0,2\tau_0),\qquad 0<4\tau_0<r_0.
	\]
	Take $\lambda$ sufficiently large that $|1+\varrho(2\lambda)|\ge1/2$.
	Bounded frequencies are covered by interior elliptic estimates.
	Finite propagation speed gives
	$[\varrho(\lambda-P)+\varrho(\lambda+P)](\chi_0u_j)
	=[1+\varrho(2\lambda)]u_j$ on $B(0,r_0)$.
	The term $\varrho(\lambda+P)(\chi_0u_j)$ is
	$O_{L^\infty}(\lambda^{-N})$ for every $N$.
	Sogge's representation \cite[Lemma~5.1.3]{Sogge}, with
	$\chi_0(y)/(1+\varrho(2\lambda))$ included in the amplitude, gives
	\begin{equation}\label{eq:sogge}
		\begin{gathered}
			u_j(x)=\lambda^{1/2}\int_Ue^{i\lambda d_g(x,y)}
			A_{\lambda,j}(x,y)u_j(y)\dd V_g(y)+R_{\lambda,j}(x),\\
			\|A_{\lambda,j}\|_{C^N}\le C_N,\qquad
			\|R_{\lambda,j}\|_{L^\infty(B(0,r_0))}\le C_N\lambda^{-N},\\
			\supp A_{\lambda,j}\subset\{(x,y):\tau_0<d_g(x,y)<2\tau_0\}.
		\end{gathered}
	\end{equation}
	After fixed finite partitions in $x,y$, choose coordinate maps
	$y_j(t,s)$ given by translations and orthogonal transformations of fixed
	bounded rectangles, and let
	\[
	\begin{gathered}
		Y_{j,t}(s)=y_j(t,s),\qquad f_{j,t}(s)=u_j(y_j(t,s)),\\
		a_{\lambda,j,t}(x,s)=A_{\lambda,j}(x,y_j(t,s))
		\sqrt{\det g(y_j(t,s))}.
	\end{gathered}
	\]
	The partition cutoffs are included in $w$ and the amplitudes.
	Extend $f_{j,t}$ by zero. For functions supported in these rectangles,
	Cauchy--Schwarz and the metric bounds give
	\begin{equation}\label{eq:iterated-input-bound}
		\int\|f(t,\cdot)\|_{L^2(ds)}\dd t
		\le C\left(\iint|f(t,s)|^2\dd t\dd s\right)^{1/2}
		\le C\|f\|_{L^2(dV_g)}.
	\end{equation}
	In particular, $\int\|f_{j,t}\|_2\dd t\le C\|u_j\|_{L^2(U)}\le C$.
	For fixed $j,t$, define
	\begin{equation}\label{eq:distance-operator}
		E_{\lambda,j,t}f(x)=\lambda^{1/2}\int
		e^{i\lambda d_g(x,Y_{j,t}(s))}a_{\lambda,j,t}(x,s)f(s)\dd s,
		\qquad \|a_{\lambda,j,t}\|_{C^N}\le C_N.
	\end{equation}
	Write $E_\lambda,Y,a_\lambda$ for one fixed pair $j,t$. Choose the
	rectangles so that, on fixed product neighborhoods of the supports,
	\[
	|\text{det}_g(Y'(s),\nabla_y d_g(x,Y(s)))|\ge c>0.
	\]
	Let $v(x,s)=-\nabla_xd_g(x,Y(s))$. Differentiating
	$Y(s)=\exp_x(d_g(x,Y(s))v(x,s))$, Gauss's lemma and the uniform
	invertibility of $d\exp_x$ give
	\begin{equation}\label{eq:direction-bounds}
		\begin{gathered}
			|v|_g=1,\qquad \langle v,\partial_sv\rangle_g=0,\\
			|\partial_sv|_g\approx
			\frac{|\text{det}_g(Y'(s),\nabla_y d_g(x,Y(s)))|}{d_g(x,Y(s))}\approx1.
		\end{gathered}
	\end{equation}
	The estimates below apply to smooth $Y$ satisfying these bounds,
	with uniformly bounded derivatives and a uniform positive distance
	from the $x$-domain.
	For $\dd t=\dd t_1\dd t_2\dd t_3$, Fubini reduces the trilinear integral
	to a fixed finite sum of
	\begin{equation}\label{eq:sogge-reduction}
		\int_{\R^3}\left(\int_Uw(x)
		\prod_{j=1}^3(E_{\lambda,j,t_j}f_{j,t_j})(x)\dd V_g(x)\right)\dd t
		+O_N(\lambda^{-N}).
	\end{equation}
	Indeed, $\|u_j\|_2+\|u_j-R_{\lambda,j}\|_2\le C$ near $\supp w$,
	so each remainder is estimated in $L^\infty$ and the other two factors
	in $L^2$. Henceforth $t_j$ are fixed and
	$E_{\lambda,j}=E_{\lambda,j,t_j}$, $Y_j=Y_{j,t_j}$, and
	$a_{\lambda,j}=a_{\lambda,j,t_j}$, with arbitrary $f_j\in L^2(\R)$
	supported in the corresponding intervals.
	
	\subsection{Two basic estimates}
	For real smooth phases $\phi(x,y)$, $x,y\in\R^d$, and smooth amplitudes
	of fixed compact support, H\"ormander's estimate
	\cite[Theorem~2.1.1]{Sogge} gives
	\begin{equation}\label{eq:nondegenerate-L2}
		|\det\phi_{xy}|\ge c>0
		\quad\Longrightarrow\quad
		\left\|\int e^{i\lambda\phi(x,y)}\chi(x,y)f(y)\dd y\right\|_2
		\le C\lambda^{-d/2}\|f\|_2,\qquad d=1,2.
	\end{equation}
	After a fixed rotation and restriction to small rectangles,
	$|\text{det}_g((1,0),v)|\ge c$. By \eqref{eq:direction-bounds},
	\[
	|\partial_{x_1}\partial_s d_g(x,Y(s))|
	=|\partial_sv|_g|\text{det}_g((1,0),v)|\ge c.
	\]
	Applying \eqref{eq:nondegenerate-L2} in $x_1,s$ at fixed $x_2$ gives
	\begin{equation}\label{eq:trace}
		\sup_{x_2}\|E_\lambda f(\cdot,x_2)\|_{L^2(dx_1)}
		+\|E_\lambda f\|_{L^2(dV_g)}\le C\|f\|_2.
	\end{equation}
	Let $\phi_j(x,s_j)=d_g(x,Y_j(s_j))$. If
	$|\text{det}_g(\nabla_x\phi_i,\nabla_x\phi_j)|\ge c$, then
	\[
	\begin{aligned}
		&|\det\partial_x\partial_{(s_i,s_j)}(\phi_i+\phi_j)|\\
		&\quad=\sqrt{\det g}\,|\partial_{s_i}\nabla_x\phi_i|_g
		|\partial_{s_j}\nabla_x\phi_j|_g
		|\text{det}_g(\nabla_x\phi_i,\nabla_x\phi_j)|\ge c.
	\end{aligned}
	\]
	Apply \eqref{eq:nondegenerate-L2} with $d=2$, phase $\phi_i+\phi_j$,
	and input $f_i\otimes f_j$. The prefactor $\lambda$ cancels its
	$\lambda^{-1}$ bound, yielding
	\begin{equation}\label{eq:bilinear}
		\|E_{\lambda,i}f_iE_{\lambda,j}f_j\|_2
		\le C\|f_i\|_2\|f_j\|_2.
	\end{equation}
	If three unit vectors sum to zero, their pairwise angles are $2\pi/3$.
	Compactness and fixed finite partitions in $x,s_1,s_2,s_3$ therefore
	reduce to either
	\begin{equation}\label{eq:angle-alternatives}
		\left|\sum_{j=1}^3\nabla_x\phi_j\right|_g\ge c
		\quad\text{or}\quad
		|\text{det}_g(\nabla_x\phi_i,\nabla_x\phi_j)|\ge c\quad(i\ne j).
	\end{equation}
	In the first case, nonstationary phase \cite[Section~1.1]{Sogge}
	gives, for fixed $b\in C_c^\infty(U)$ and every $N$,
	\begin{equation}\label{eq:nonstationary}
		\left|\int b\prod_{j=1}^3E_{\lambda,j}f_j\dd V_g\right|
		\le C_{b,N}\lambda^{-N}\prod_{j=1}^3\|f_j\|_2.
	\end{equation}
	In the second case, \eqref{eq:trace} and \eqref{eq:bilinear} give
	$|\int b\prod_{j=1}^3E_{\lambda,j}f_j\dd V_g|
	\le C\|b\|_\infty\prod_{j=1}^3\|f_j\|_2$.
	
	\begin{lemma}\label{lem:small-weight}
		For $b\in C_c^\infty(U)$ with $\supp b\subset\supp w$,
		\begin{equation}\label{eq:small-weight}
			\left|\int_Ub u_1u_2u_3\dd V_g\right|
			\le C\|b\|_{L^4(U)}+C_b\lambda^{-1/2},
		\end{equation}
		where $C$ is independent of the derivatives of $b$.
	\end{lemma}
	\begin{proof}
		Use \eqref{eq:nonstationary} in the first case of
		\eqref{eq:angle-alternatives}. In the second, H\"older gives
		\[
		\left|\int b\prod_{j=1}^3E_{\lambda,j}f_j\dd V_g\right|
		\le\|b\|_4\prod_{1\le i<j\le3}\|E_{\lambda,i}f_iE_{\lambda,j}f_j\|_2^{1/2}
		\le C\|b\|_4\prod_{j=1}^3\|f_j\|_2.
		\]
		Now use \eqref{eq:iterated-input-bound} and \eqref{eq:sogge-reduction}.
	\end{proof}
	
	\subsection{Rescaling and summation}\label{sec:local-inputs}
	\begin{lemma}\label{lem:radial}
		For $E_\lambda$ in \eqref{eq:distance-operator}, choose $\rho_0>0$
		so that the preceding geometric bounds hold for
		$d_g(x,\supp w)<16\rho_0$, with $16\rho_0<\inf_{x,s}d_g(x,Y(s))$ there.
		Fix $\chi_1\in C_c^\infty(U)$ supported where $d_g(y,\supp w)<8\rho_0$
		and equal to one where $d_g(y,\supp w)\le6\rho_0$, and define
		\[
		V_\lambda(y)=\chi_1(y)\lambda^{1/2}
		\int e^{i\lambda d_g(y,Y(s))}f(s)\dd s.
		\]
		For $0<\rho\le\rho_0$ and $\lambda\rho\ge1$, there is a smooth amplitude
		$b_{\lambda,\rho}$ such that, near $\supp w$,
		\begin{equation}\label{eq:radial-factorization}
			E_\lambda f(x)=\lambda^{1/2}\rho^{-3/2}
			\int e^{i\lambda d_g(x,y)}b_{\lambda,\rho}(x,y)V_\lambda(y)\dd V_g(y)
			+R_{\lambda,\rho}(x),
		\end{equation}
		with
		\[
		\supp b_{\lambda,\rho}\subset\{\rho<d_g(x,y)<2\rho\},\qquad
		|\partial_x^\alpha\partial_y^\beta b_{\lambda,\rho}|
		\le C_{\alpha,\beta}\rho^{-|\alpha|-|\beta|},
		\]
		and
		\begin{equation}\label{eq:radial-remainder}
			\|R_{\lambda,\rho}\|_{L^\infty(\supp w)}
			\le C\lambda^{-1/2}\rho^{-1}\|f\|_2.
		\end{equation}
		The function $V_\lambda$ is independent of $\rho$ and satisfies
		\eqref{eq:trace} with the fixed amplitude $\chi_1(y)$ and a fixed cutoff
		in $s$ equal to one on $\supp f$. On $\supp b_{\lambda,\rho}$,
		$-\nabla_xd_g(x,y)=v(x,s)$ for some $s$ in the fixed interval.
	\end{lemma}
	\begin{proof}
		For fixed $x$, let
		\[
		\begin{gathered}
			y_\rho(\tau,\sigma)=\exp_x(\rho\tau v(x,\sigma)),\qquad 1<\tau<2,\\
			D_\rho(\tau,\sigma,s)=
			\frac{\rho\tau+d_g(y_\rho(\tau,\sigma),Y(s))-d_g(x,Y(s))}{\rho}.
		\end{gathered}
		\]
		Gauss's lemma and Taylor's formula give, uniformly in $x,\tau,\sigma,s$, 
		\[
		\begin{gathered}
			dV_g(y_\rho)=\rho|\partial_\sigma y_\rho|_g\dd\tau\dd\sigma,
			\qquad \rho^{-1}|\partial_\sigma y_\rho|_g\approx1,\\
			D_\rho=\tau(1-\langle v(x,\sigma),v(x,s)\rangle_g)
			+O_{C^N}(\rho).
		\end{gathered}
		\]
		The triangle inequality is an equality at $\sigma=s$. Thus
		\[
		D_\rho(\tau,s,s)=\partial_\sigma D_\rho(\tau,s,s)=0,\qquad
		\partial_\sigma^2D_\rho(\tau,s,s)
		=\tau|\partial_sv(x,s)|_g^2+O(\rho)\ge c.
		\]
		Take $\eta_1\in C_c^\infty((1,2))$ with
		$\int\eta_1=1$, and define $b_{\lambda,\rho}$ by
		\[
		b_{\lambda,\rho}(x,y_\rho(\tau,s))
		\frac{|\partial_sy_\rho(\tau,s)|_g}{\rho}
		=\frac{e^{-i\pi/4}}{\sqrt{2\pi}}\eta_1(\tau)a_\lambda(x,s)
		\sqrt{\partial_\sigma^2D_\rho(\tau,s,s)}.
		\]
		Extend it by zero. The map
		$(\tau,\sigma)\mapsto(y_\rho(\tau,\sigma)-x)/\rho$ and its inverse
		have uniform smooth bounds, giving the asserted derivative bounds.
		On sufficiently small fixed intervals, $\partial_\sigma^2D_\rho\ge c$.
		One-dimensional stationary phase
		\cite[Corollary~1.1.8 and (1.1.20)]{Sogge} gives
		\[
		\begin{aligned}
			&\int e^{i\lambda\rho D_\rho}
			b_{\lambda,\rho}(x,y_\rho(\tau,\sigma))
			\frac{|\partial_\sigma y_\rho(\tau,\sigma)|_g}{\rho}\dd\sigma\\
			&\quad=(\lambda\rho)^{-1/2}\eta_1(\tau)a_\lambda(x,s)
			+O((\lambda\rho)^{-3/2}).
		\end{aligned}
		\]
		The error is uniform in $x,\tau,s$. Since $\int\eta_1=1$,
		\[
		\begin{aligned}
			&\lambda\rho^{-3/2}\int
			e^{i\lambda[d_g(x,y)+d_g(y,Y(s))]}b_{\lambda,\rho}(x,y)\dd V_g(y)\\
			&\quad=\lambda^{1/2}e^{i\lambda d_g(x,Y(s))}a_\lambda(x,s)
			+O(\lambda^{-1/2}\rho^{-1}),
		\end{aligned}
		\]
		where $\lambda\rho^{-3/2}\rho^2(\lambda\rho)^{-1/2}=\lambda^{1/2}$.
		For $x$ near $\supp w$, $\chi_1=1$ on the integration set.
		Integrate against $f(s)$ and use $\|f\|_1\le C\|f\|_2$.
	\end{proof}
	
	Use $V_{\lambda,j},b_{\lambda,\rho,j},R_{\lambda,\rho,j}$ for the three
	operators. Cover $\supp w$ by $B(p_\ell,c_0\rho)$, with centers in
	$\supp w$, fixed $0<c_0<1/16$, and
	$\sum_\ell\mathbf1_{B(p_\ell,4\rho)}\le C$. On $\supp b_{\lambda,\rho,j}$,
	\[
	x\in B(p_\ell,c_0\rho)
	\quad\Longrightarrow\quad
	|y-p_\ell|\le c_0\rho+\sqrt2\,d_g(x,y)<3\rho.
	\]
	Define
	\begin{equation}\label{eq:local-inputs}
		\begin{gathered}
			g_\ell(z)=g(p_\ell+\rho z),\qquad \lambda_\ell=\lambda\rho,
			\qquad K_{g_\ell}(z)=\rho^2K_g(p_\ell+\rho z),\\
			f_{j,\ell}(z)=\rho V_{\lambda,j}(p_\ell+\rho z)\mathbf1_{B(0,3)}(z),\qquad
			m_{j,\ell}=\|f_{j,\ell}\|_{L^2(dV_{g_\ell})}
			=\|V_{\lambda,j}\|_{L^2(B(p_\ell,3\rho),g)}.
		\end{gathered}
	\end{equation}
	Here $(g_\ell)_{ab}(z)=g_{ab}(p_\ell+\rho z)$. Since
	$d_g(p_\ell+\rho x,p_\ell+\rho y)=\rho d_{g_\ell}(x,y)$ and
	$dV_g=\rho^2dV_{g_\ell}$,
	\eqref{eq:radial-factorization} becomes
	\begin{equation}\label{eq:child-operators}
		\begin{gathered}
			E_{\lambda,j}f_j(p_\ell+\rho x)
			=\rho^{-1}Q_{j,\ell}f_{j,\ell}(x)+R_{\lambda,\rho,j}(p_\ell+\rho x),\\
			Q_{j,\ell}f(x)=\lambda_\ell^{1/2}\int e^{i\lambda_\ell d_{g_\ell}(x,y)}
			b_{\lambda,\rho,j}(p_\ell+\rho x,p_\ell+\rho y)f(y)\dd V_{g_\ell}(y).
		\end{gathered}
	\end{equation}
	The amplitudes have uniform smooth bounds. As in
	Section~\ref{sec:sogge}, translations and orthogonal changes of the
	$y$ coordinates give the form \eqref{eq:distance-operator} for the inner
	$s$ integral, with $\sqrt{\det g_\ell(y)}$ included in the amplitude. Equation \eqref{eq:iterated-input-bound}
	therefore transfers trilinear estimates to $Q_{j,\ell}$ with factors
	$m_{j,\ell}$. All geodesics joining points in the $x$ and $y$ supports lie in $B(0,4)$.
	
	Assume the second condition in \eqref{eq:angle-alternatives}.
	For sufficiently small $\rho_0$, the determinant in
	\eqref{eq:bilinear} stays bounded below under $|z|\le6\rho_0$.
	Applying that estimate to $d_g(x,Y_1(s_1))+d_g(x+z,Y_2(s_2))$ gives
	\[
	\|V_{\lambda,1}(x)V_{\lambda,2}(x+z)\|_{L^2(dx)}
	\le C\|f_1\|_2\|f_2\|_2.
	\]
	Bounded overlap and $y=x+z$ yield
	\begin{equation}\label{eq:pair-input-sum}
		\sum_\ell m_{1,\ell}^2m_{2,\ell}^2
		\le C\int_{|z|\le6\rho}\int
		|V_{\lambda,1}(x)V_{\lambda,2}(x+z)|^2\dd x\dd z
		\le C\rho^2\|f_1\|_2^2\|f_2\|_2^2.
	\end{equation}
	Also $\sum_\ell m_{3,\ell}^2\le C\|V_{\lambda,3}\|_2^2\le C\|f_3\|_2^2$.
	Consequently,
	\begin{equation}\label{eq:area-input-sum}
		\sum_\ell\prod_{j=1}^3m_{j,\ell}
		\le\Bigl(\sum_\ell m_{1,\ell}^2m_{2,\ell}^2\Bigr)^{1/2}
		\Bigl(\sum_\ell m_{3,\ell}^2\Bigr)^{1/2}
		\le C\rho\prod_{j=1}^3\|f_j\|_2.
	\end{equation}
	
	For the final estimate assume $\lambda^{1/2}\rho\ge1$.
	Equations \eqref{eq:trace} and \eqref{eq:radial-remainder} give, on $\supp w$,
	\[
	\|E_{\lambda,j}f_j\|_2+
	\|E_{\lambda,j}f_j-R_{\lambda,\rho,j}\|_2\le C\|f_j\|_2,
	\qquad \|R_{\lambda,\rho,j}\|_\infty
	\le C\lambda^{-1/2}\rho^{-1}\|f_j\|_2.
	\]
	Let $\supp w_\ell\subset B(p_\ell,c_0\rho)$ and
	$\|\sum_\ell w_\ell\|_\infty\le C$. Replace each factor
	$E_{\lambda,j}f_j$ by $E_{\lambda,j}f_j-R_{\lambda,\rho,j}$ in turn.
	Each difference has one remainder in $L^\infty$ and two factors in $L^2$.
	Since the remainders are independent of $\ell$, estimate their sum with
	weight $\sum_\ell w_\ell$. Finally \eqref{eq:child-operators} and
	$\rho^2(\rho^{-1})^3=\rho^{-1}$ give
	\begin{equation}\label{eq:localization-bound}
		\begin{aligned}
			\left|\int\left(\sum_\ell w_\ell\right)
			\prod_{j=1}^3E_{\lambda,j}f_j\dd V_g\right|\le\rho^{-1}\sum_\ell
			\left|\int w_\ell(p_\ell+\rho x)
			\prod_{j=1}^3(Q_{j,\ell}f_{j,\ell})(x)\dd V_{g_\ell}(x)\right|
			+\frac C{\lambda^{1/2}\rho}\prod_{j=1}^3\|f_j\|_2.
		\end{aligned}
	\end{equation}
	
	\section{Curvature estimates and the upper bounds}\label{sec:upper-proofs}
	
	A curvature identity for the reduced phase
	is the key geometric ingredient in the proof of Theorem~\ref{thm:upper}.
	We first recall the two-sided fold theorem of Pan and Sogge, from
	which the required trilinear oscillatory estimate follows.
	
	\begin{theorem}[Pan--Sogge {\cite[Theorem~1.1]{PanSogge}}]
		\label{thm:pan-sogge}
		Let $\phi(x,y)$ be real and smooth, $x,y\in\R^2$, and let
		$a\in C_c^\infty(\R^2\times\R^2)$. Let $h=\det\phi_{xy}$.
		Suppose $|\phi_{x_1y_1}|\ge c>0$ and
		\[
		\left|\left(\partial_{y_2}
		-\frac{\phi_{x_1y_2}}{\phi_{x_1y_1}}\partial_{y_1}\right)h\right|\ge c,
		\qquad
		\left|\left(\partial_{x_2}
		-\frac{\phi_{x_2y_1}}{\phi_{x_1y_1}}\partial_{x_1}\right)h\right|\ge c
		\]
		on a fixed neighborhood of $\supp a$. Then, for $\mu\ge1$,
		\[
		\left\|\int_{\R^2}e^{i\mu\phi(x,y)}a(x,y)f(y)\dd y\right\|_{L^2(dx)}
		\le C\mu^{-5/6}\|f\|_{L^2(dy)}.
		\]
		The constant depends only on $c$, the fixed support, and finitely many
		smooth bounds for $\phi$ and $a$.
	\end{theorem}
	
	 On \(h=0\), the displayed vector fields span the right and left kernels of \(\phi_{xy}\), respectively. Their derivatives of \(h\) give the two-sided fold conditions; see Greenleaf and Seeger \cite[(2.1)]{GreenleafSeeger2002}. See also Melrose and Taylor \cite{MelroseTaylor}, Cuccagna \cite{Cuccagna}, and Hassell and Tacy \cite{HassellTacy}.
	
	\subsection{A trilinear estimate}
	The following lemma applies to the reduced phase obtained by stationary
	phase in the spatial variables.
	
	For a real phase $\Phi(s_1,s_2,s_3)$, write
	$\partial_j=\partial/\partial s_j$ and $\Phi_{ij}=\partial_i\partial_j\Phi$.
	When $\Phi_{12}\Phi_{13}\Phi_{23}\ne0$, let
	\begin{equation}\label{eq:phase-data}
		q_\Phi=\Phi_{33}-\frac{\Phi_{13}\Phi_{23}}{\Phi_{12}}.
	\end{equation}
	\begin{lemma}\label{lem:trilinear}
		On a sufficiently small fixed product of intervals, suppose
		\[
		\begin{gathered}
			\Phi_\varepsilon(s)=H(s_1+s_2+s_3)+\varepsilon\Psi_\varepsilon(s),
			\qquad |H''|\ge c,\\
			\left|(\Phi_{23}\partial_1-\Phi_{13}\partial_2)q_\Phi\right|
			\ge c\varepsilon,\qquad \Phi=\Phi_\varepsilon,
		\end{gathered}
		\]
		where $H,\Psi_\varepsilon$ are real with uniform smooth bounds.
		For $\lambda\ge1$, $0<\varepsilon\le\varepsilon_0$, with $\varepsilon_0$
		fixed and sufficiently small, smooth amplitudes $a_\lambda$ of fixed
		compact support and uniform smooth bounds, and $f_j\in L^2(\R)$ supported
		in the corresponding intervals,
		\begin{equation}\label{eq:uniform-trilinear}
			\left|\int e^{i\lambda\Phi_\varepsilon(s)}a_\lambda(s)
			\prod_{j=1}^3 f_j(s_j)\dd s\right|
			\le C\lambda^{-1/2}(1+\lambda\varepsilon)^{-1/6}\prod_{j=1}^3\|f_j\|_2.
		\end{equation}
	\end{lemma}
	\begin{proof}
		If $\lambda\varepsilon\le1$, apply \eqref{eq:nondegenerate-L2} in
		$(s_1,s_2)$ for fixed $s_3$ and use $\|f_3\|_1\le C\|f_3\|_2$.
		Assume $\lambda\varepsilon\ge1$ and write $s=(x,y,z)$, $\Phi=\Phi_\varepsilon$.
		Define
		\[
		\begin{aligned}
			K_\lambda(x,\eta,u,y)
			&=\int e^{i\lambda[\Phi(x,y,z)-\Phi(u+y,\eta,z)]}
			a_\lambda(x,y,z)\overline{a_\lambda(u+y,\eta,z)}\dd z,\\
			\mathcal K_\lambda h(x,\eta)
			&=\iint K_\lambda(x,\eta,u,y)h(u,y)\dd u\dd y.
		\end{aligned}
		\]
		Cauchy--Schwarz in $z$, followed by expansion of the square and the
		substitution $u=\xi-y$, gives
		\begin{equation}\label{eq:crossed-identity}
			\begin{aligned}
				\left|\int e^{i\lambda\Phi(s)}a_\lambda(s)
				\prod_{j=1}^3f_j(s_j)\dd s\right|^2&\le\|f_3\|_2^2\iint\mathcal K_\lambda
				[f_2(y)\overline{f_1(u+y)}](x,\eta)
				f_1(x)\overline{f_2(\eta)}\dd x\dd\eta\\
				&\le\|\mathcal K_\lambda\|_{L^2\to L^2}\prod_{j=1}^3\|f_j\|_2^2.
			\end{aligned}
		\end{equation}
		Take the unitary Fourier transform in $x$ and multiply by
		$e^{i\lambda p\eta}$, obtaining the kernel
		\[
		\widehat K_\lambda(p,\eta,u,y)
		=\left(\frac\lambda{2\pi}\right)^{1/2}
		\int e^{-i\lambda p(x-\eta)}K_\lambda(x,\eta,u,y)\dd x.
		\]
		A tilde denotes evaluation at $(u+y,\eta,z)$. The critical point equations in $x,z$ and the Hessian in these
		variables are
		\[
		\Phi_1=p,\qquad \Phi_3=\widetilde\Phi_3,\qquad
		H_*=\begin{pmatrix}\Phi_{11}&\Phi_{13}\\
			\Phi_{13}&\Phi_{33}-\widetilde\Phi_{33}\end{pmatrix},
		\qquad h_*=\det H_*.
		\]
		On $y=\eta$, $x=u+\eta$ and $h_*=-\Phi_{13}^2$.
		On fixed smaller parameter sets, $h_*\le-c$ and the stationary point
		$(x_*,z_*)$ is unique and smooth. Let
		\[
		\theta_\varepsilon(p,\eta,u,y)
		=\left[\Phi(x,y,z)-\Phi(u+y,\eta,z)-p(x-\eta)\right]_{(x,z)=(x_*,z_*)}.
		\]
		Stationary phase \cite[Corollary~1.1.8 and (1.1.20)]{Sogge} on a fixed
		enlargement of the amplitude support, with the amplitude extended by zero, gives
		\begin{equation}\label{eq:transformed-kernel}
			\widehat K_\lambda
			=\lambda^{-1/2}e^{i\lambda\theta_\varepsilon}B_\lambda+r_\lambda,
			\qquad \|B_\lambda\|_{C^N}\le C_N,\qquad
			\|r_\lambda\|_{L^2(dp\,d\eta\,du\,dy)}\le C\lambda^{-3/2}.
		\end{equation}
		Here $B_\lambda$ has fixed compact support. On the complementary region, where the phase is uniformly nonstationary,
		integration by parts gives rapid decay in $\lambda$ and $1+|p|$.
		
		In the following identities, evaluate untilded derivatives at
		$(x_*,y,z_*)$ and tilded derivatives at $(u+y,\eta,z_*)$. Since
		\[
		\begin{gathered}
			\partial_p\theta_\varepsilon=-x_*+\eta,\qquad
			\partial_\eta\theta_\varepsilon=-\widetilde\Phi_2+p,\\
			H_*\binom{\partial_u x_*}{\partial_u z_*}
			=\binom0{\widetilde\Phi_{13}},\qquad
			H_*\binom{\partial_y x_*}{\partial_y z_*}
			=\binom{-\Phi_{12}}{\widetilde\Phi_{13}-\Phi_{23}},
		\end{gathered}
		\]
		the matrix $\mathcal M=\partial_{(p,\eta)}\partial_{(u,y)}\theta_\varepsilon$
		satisfies
		\begin{equation}\label{eq:crossed-determinant}
			\begin{gathered}
				\mathcal M_{11}=\frac{\Phi_{13}\widetilde\Phi_{13}}{h_*},\qquad
				\det\mathcal M=\frac{\Phi_{12}\widetilde\Phi_{12}}{h_*}
				(q_\Phi-\widetilde q_\Phi),\\
				\left.\partial_y\det\mathcal M\right|_{y=\eta}
				=\frac{\Phi_{12}^2}{\Phi_{13}^3}
				(\Phi_{23}\partial_1-\Phi_{13}\partial_2)q_\Phi.
			\end{gathered}
		\end{equation}
		For the last identity, use $\partial_yx_*=1-\Phi_{23}/\Phi_{13}$ on
		$y=\eta$. The $\partial_yz_*$ terms in
		$\partial_y(q_\Phi-\widetilde q_\Phi)$ cancel.
		
		At $\varepsilon=0$, $\theta_0=-pu$. Therefore, on fixed smaller sets,
		\begin{equation}\label{eq:crossed-normal-form}
			\begin{gathered}
				\theta_\varepsilon=-pu+\varepsilon\psi_\varepsilon,\qquad
				\det\mathcal M=\varepsilon(y-\eta)e_\varepsilon,\\
				|e_\varepsilon|\ge c,\qquad
				\|\psi_\varepsilon\|_{C^N}+\|e_\varepsilon\|_{C^N}\le C_N.
			\end{gathered}
		\end{equation}
		On $y=\eta$, the determinant vanishes and
		\eqref{eq:crossed-determinant} gives
		$|\partial_y(\varepsilon^{-1}\det\mathcal M)|\ge c$.
		
		Set $\mu=\lambda\varepsilon$. Choose
		$\vartheta\in C_c^\infty((-2,2))$ with
		$\sum_{k\in\mathbb Z}\vartheta(p/\varepsilon-k)=1$, and let
		$\mathcal K_{\lambda,k}$ have kernel
		\[
		\lambda^{-1/2}\vartheta(p/\varepsilon-k)
		e^{i\lambda\theta_\varepsilon(p,\eta,u,y)}B_\lambda(p,\eta,u,y).
		\]
		By \eqref{eq:transformed-kernel} and unitarity, it remains to estimate
		$\sum_k\mathcal K_{\lambda,k}$, up to an $O(\lambda^{-3/2})$ error in
		$L^2\to L^2$.
		In the $k$th term substitute $p=\varepsilon(k+p')$. After multiplying
		$h(u,y)$ by $e^{-i\mu ku}$, the phase at frequency $\mu$ is
		\[
		\Theta_k(p',\eta,u,y)=-p'u+\psi_\varepsilon(\varepsilon(k+p'),\eta,u,y),
		\qquad |p'|<2.
		\]
		The phase and amplitude have uniform smooth bounds. Let
		$h_k=\det(\Theta_k)_{(p',\eta),(u,y)}$:
		\[
		|(\Theta_k)_{p'u}|\ge\tfrac12,\qquad
		h_k=(y-\eta)e_\varepsilon(\varepsilon(k+p'),\eta,u,y).
		\]
		Writing $e_\varepsilon$ for its value at $(\varepsilon(k+p'),\eta,u,y)$, we obtain
		\[
		\begin{aligned}
			\left(\partial_y-\frac{(\Theta_k)_{p'y}}{(\Theta_k)_{p'u}}\partial_u\right)h_k
			&=e_\varepsilon+O(|y-\eta|),\\
			\left(\partial_\eta-\frac{(\Theta_k)_{\eta u}}{(\Theta_k)_{p'u}}\partial_{p'}\right)h_k
			&=-e_\varepsilon+O(|y-\eta|).
		\end{aligned}
		\]
		Since $|e_\varepsilon|\ge c$, both absolute values are at least $c/2$
		for small fixed $|y-\eta|$, uniformly in $k,\varepsilon$.
		Theorem~\ref{thm:pan-sogge} and $\dd p=\varepsilon\dd p'$ give
		\[
		\|\mathcal K_{\lambda,k}\|_{L^2\to L^2}
		\le C\lambda^{-1/2}\varepsilon^{1/2}\mu^{-5/6}
		=C\lambda^{-1}\mu^{-1/3}.
		\]
		For $|p/\varepsilon-k|,|\widetilde p/\varepsilon-\ell|<2$ and
		$|k-\ell|>C_0$,
		\[
		\bigl|\partial_u\{\lambda[\theta_\varepsilon(p,\eta,u,y)
		-\theta_\varepsilon(\widetilde p,\widetilde\eta,u,y)]\}\bigr|
		\ge c\mu|k-\ell|.
		\]
		Three integrations by parts in $u$ bound the kernel of
		$\mathcal K_{\lambda,k}\mathcal K_{\lambda,\ell}^*$ by
		$C\lambda^{-1}\mu^{-3}(1+|k-\ell|)^{-3}$.
		The $p,\widetilde p$ supports have lengths $O(\varepsilon)$.
		Schur's test and the preceding norm bound therefore give, for all $k,\ell$,
		\[
		\|\mathcal K_{\lambda,k}\mathcal K_{\lambda,\ell}^*\|_{L^2\to L^2}
		\le C\lambda^{-2}\mu^{-2/3}(1+|k-\ell|)^{-2}.
		\]
		Bounded overlap of the $p$ intervals and Schur's test on this matrix yield
		\[
		\|\mathcal K_\lambda\|_{L^2\to L^2}
		\le C\left(\sup_k\sum_\ell
		\|\mathcal K_{\lambda,k}\mathcal K_{\lambda,\ell}^*\|_{L^2\to L^2}\right)^{1/2}
		+C\lambda^{-3/2}
		\le C\lambda^{-1}\mu^{-1/3}.
		\]
		Taking the square root in \eqref{eq:crossed-identity} proves the lemma.
	\end{proof}
	
\subsection{Fermat point and Gaussian curvature}
For the following lemma, use a fixed geodesically convex normal coordinate
neighborhood of $0$. Choose positive $t_j\approx1$ and Euclidean unit vectors $v_j^0$
with $v_1^0+v_2^0+v_3^0=0$. Put $s=(s_1,s_2,s_3)$ and
$v^\perp=(-v_2,v_1)$ for $v=(v_1,v_2)$, and define, for $j=1,2,3$,
\[
Y_j(s_j)=t_jv_j^0+s_j(v_j^0)^\perp,\qquad
\phi_j(x,s_j)=d_g(x,Y_j(s_j)),\qquad v_j=-\nabla_x\phi_j.
\]
Restrict $x,s_j$ to small fixed neighborhoods of $0$, with all defining
geodesics contained in the coordinate neighborhood. Choose $n_j$ so that
$(v_j,n_j)$ is a positively oriented $g$-orthonormal basis. For smooth
$f=f(x,s_j)$, write $v_j(f)=\langle\nabla_xf,v_j\rangle_g$, keeping $s_j$ fixed.

\begin{lemma}\label{lem:kernel-phase}
	Suppose $K_g=\varepsilon F$, $0<\varepsilon\le\varepsilon_0$,
	$|F|\ge c>0$, and $\|F\|_{C^N}\le C_N$ on this neighborhood.
	For $\lambda\ge1$ and smooth $a_\lambda(x,s)$ of fixed compact support
	and uniform smooth bounds, there are a real phase $\Phi$ and smooth
	$b_\lambda$ such that
	\begin{equation}\label{eq:stationary-kernel}
		\lambda^{3/2}\int e^{i\lambda(\phi_1+\phi_2+\phi_3)}
		a_\lambda(x,s)\dd V_g(x)
		=\lambda^{1/2}e^{i\lambda\Phi(s)}b_\lambda(s)+O(\lambda^{-1/2}),
	\end{equation}
	where $\|b_\lambda\|_{C^N}\le C_N$ and $\Phi$ satisfies
	Lemma~\ref{lem:trilinear}. All bounds are uniform in $t_1,t_2,t_3$.
\end{lemma}
\begin{proof}
	Let $I$ be the Euclidean metric and $J_g=\sqrt{\det g}$.
	The radial Jacobi equation and Gauss's lemma
	\cite[Section~3.2.4 and Example~3.2.12]{Petersen} give
	\[
	\begin{gathered}
		J_g(x)-1=-\varepsilon|x|^2\int_0^1t(1-t)F(tx)J_g(tx)\dd t,\\
		g-I=(J_g+1)\frac{J_g-1}{|x|^2}(|x|^2I-x\otimes x).
	\end{gathered}
	\]
	The quotient extends smoothly to $0$, and hence
	\begin{equation}\label{eq:flat-comparison}
		\|g-I\|_{C^N}\le C_N\varepsilon.
	\end{equation}
	Since $|\nabla_x\phi_j|_g=1$, there are smooth $\alpha_j,\beta_j$ with
	\[
	\nabla_x^2\phi_j=\alpha_jn_j\otimes n_j,\qquad
	\partial_{s_j}v_j=\beta_jn_j,\qquad c\le\alpha_j,\beta_j\le C.
	\]
	Here the bounds follow from \eqref{eq:flat-comparison}.
	At $(x,s)=(0,0)$, $v_j=v_j^0$ and
	$\nabla_x^2(\phi_1+\phi_2+\phi_3)\ge cI$.
	The implicit function theorem therefore gives a unique smooth $x_c(s)$ near $0$ with
	\begin{equation}\label{eq:reduced-phase}
		\left.\nabla_x(\phi_1+\phi_2+\phi_3)\right|_{x=x_c(s)}=0,
		\qquad x_c(0)=0.
	\end{equation}
	Thus $x_c(s)$ is the local Fermat point of $Y_1(s_1),Y_2(s_2),Y_3(s_3)$:
	it minimizes $\phi_1+\phi_2+\phi_3$ near $0$, and the three joining
	geodesics meet there at angles of $2\pi/3$.
	
	Define
	$\Phi(s)=\phi_1(x_c(s),s_1)+\phi_2(x_c(s),s_2)+\phi_3(x_c(s),s_3)$.
	In formulas involving only $s$, evaluate $v_j,n_j,\alpha_j,\beta_j$
	at $(x_c(s),s_j)$. Since $v_1+v_2+v_3=0$, we may interchange $1,2$ so that
	\[
	n_1=-\frac{\sqrt3}{2}v_3-\frac12n_3,\qquad
	n_2=\frac{\sqrt3}{2}v_3-\frac12n_3.
	\]
	The Hessian $H_g=\nabla_x^2(\phi_1+\phi_2+\phi_3)|_{x=x_c}$ has the
	following matrix in the basis $(v_3,n_3)$:
	\begin{equation}\label{eq:spatial-hessian}
		\begin{gathered}
			H_g=\frac14\begin{pmatrix}
				3(\alpha_1+\alpha_2)&\sqrt3(\alpha_1-\alpha_2)\\
				\sqrt3(\alpha_1-\alpha_2)&\alpha_1+\alpha_2+4\alpha_3
			\end{pmatrix},\\
			\det H_g=\frac34(\alpha_1\alpha_2+\alpha_1\alpha_3+\alpha_2\alpha_3)\ge c.
		\end{gathered}
	\end{equation}
	Stationary phase \cite[Corollary~1.1.8 and (1.1.20)]{Sogge}
	gives \eqref{eq:stationary-kernel} with
	$b_\lambda(s)=2\pi i\,a_\lambda(x_c,s)/\sqrt{\det H_g}$.
	
	Differentiating \eqref{eq:reduced-phase} gives
	$H_g\partial_jx_c=\beta_jn_j$. Differentiating $\Phi$ and using
	\eqref{eq:spatial-hessian}, we obtain
	\begin{equation}\label{eq:geometric-mixed}
		\begin{gathered}
			\Phi_{12}=\frac{3\alpha_3\beta_1\beta_2}{4\det H_g}\approx1,\qquad
			\Phi_{13}=\frac{3\alpha_2\beta_1\beta_3}{4\det H_g}\approx1,\qquad
			\Phi_{23}=\frac{3\alpha_1\beta_2\beta_3}{4\det H_g}\approx1,\\
			\Phi_{33}=\partial_{s_3}^2\phi_3(x_c,s_3)
			-\frac{3\beta_3^2(\alpha_1+\alpha_2)}{4\det H_g},\\
			q_\Phi=\Phi_{33}-\frac{\Phi_{13}\Phi_{23}}{\Phi_{12}}
			=\left[\partial_{s_3}^2\phi_3-\frac{\beta_3^2}{\alpha_3}\right]_{x=x_c}.
		\end{gathered}
	\end{equation}
	The same equations give
	\begin{equation}\label{eq:critical-direction}
		\begin{aligned}
			H_g(\Phi_{23}\partial_1x_c-\Phi_{13}\partial_2x_c)
			&=\frac{3\beta_1\beta_2\beta_3}{4\det H_g}(\alpha_1n_1-\alpha_2n_2)\\
			&=-\frac{\sqrt3\,\beta_1\beta_2\beta_3}{2\det H_g}H_gv_3.
		\end{aligned}
	\end{equation}
	Before evaluation at $x_c$, the Riccati equation
	\cite[Corollary~3.2.10]{Petersen} and the $s_3$-derivatives of
	$\nabla_{v_3}v_3=0$ and $|\nabla_x\phi_3|_g^2=1$ give
	\begin{equation}\label{eq:riccati-identities}
		v_3(\alpha_3)=\alpha_3^2+K_g,\qquad
		v_3(\beta_3)=\alpha_3\beta_3,\qquad
		v_3(\partial_{s_3}^2\phi_3)=\beta_3^2.
	\end{equation}
	Consequently,
	\begin{equation}\label{eq:directional-curvature}
		v_3\left(\partial_{s_3}^2\phi_3-\frac{\beta_3^2}{\alpha_3}\right)
		=\frac{\beta_3^2}{\alpha_3^2}K_g.
	\end{equation}
	Since $H_g$ is invertible and
	$\Phi_{23}\partial_1-\Phi_{13}\partial_2$ leaves $s_3$ fixed,
	\eqref{eq:geometric-mixed}, \eqref{eq:critical-direction}, and
	\eqref{eq:directional-curvature} give, by the chain rule,
	\begin{align}
		(\Phi_{23}\partial_1-\Phi_{13}\partial_2)q_\Phi
		&=-\frac{\sqrt3\,\beta_1\beta_2\beta_3^3}{2\alpha_3^2\det H_g}K_g(x_c),
		\label{eq:curvature-identity}\\
		\left|(\Phi_{23}\partial_1-\Phi_{13}\partial_2)q_\Phi\right|
		&\approx\varepsilon.\label{eq:curvature-nondegeneracy}
	\end{align}
	
	For $g=I$, denote the reduced phase by $\Phi_0$ and choose
	$O_j\in SO(2)$ with $O_jv_j^0=v_1^0$.
	The $2\pi/3$ angles and continuity from $s=0$ imply that
	$O_j(Y_j-x_c)$ have the same direction. Since $O_1+O_2+O_3=0$,
	\begin{equation}\label{eq:flat-phase}
		\begin{gathered}
			\Phi_0(s)=|O_1Y_1+O_2Y_2+O_3Y_3|=H(s_1+s_2+s_3),\\
			H(\sigma)=\sqrt{(t_1+t_2+t_3)^2+\sigma^2},\qquad
			H''(\sigma)=\frac{(t_1+t_2+t_3)^2}{H(\sigma)^3}\approx1.
		\end{gathered}
	\end{equation}
	By \eqref{eq:flat-comparison} and the implicit function theorem,
	\[
	\Phi=\Phi_0+\varepsilon\Psi_\varepsilon,\qquad
	\|\Psi_\varepsilon\|_{C^N}\le C_N.
	\]
	In particular, $\Phi_{ij}\approx1$ for all $1\le i,j\le3$.
	Together with \eqref{eq:curvature-nondegeneracy}, this verifies
	Lemma~\ref{lem:trilinear}.
\end{proof}

\begin{lemma}\label{lem:nonzero-operator}
	Suppose $K_g=\varepsilon F$ on a fixed neighborhood of $\supp w$, where
	$0<\varepsilon\le1$, $|F|\ge c>0$, and $\|F\|_{C^N}\le C_N$.
	For \eqref{eq:distance-operator} and weights of fixed compact support and
	uniform smooth bounds,
	\begin{equation}\label{eq:nonzero-operator}
		\left|\int w\prod_{j=1}^3 E_{\lambda,j}f_j\dd V_g\right|
		\le C(1+\lambda\varepsilon)^{-1/6}\prod_{j=1}^3\|f_j\|_2.
	\end{equation}
\end{lemma}
\begin{proof}
	By homogeneity, assume $\|f_j\|_2=1$. The first case of
	\eqref{eq:angle-alternatives} is $O(\lambda^{-2})$ by
	\eqref{eq:nonstationary}. Hence assume the second case.
	
	Fix $0<\rho\le\rho_0$ with $\rho^2\le\varepsilon_0$ so that the curvature
	assumptions hold on $\{y:d_g(y,\supp w)<8\rho\}$.
	Equations \eqref{eq:trace} and \eqref{eq:bilinear} cover $\lambda\rho^2<1$,
	so assume $\lambda\rho^2\ge1$.
	Use Section~\ref{sec:local-inputs} to write $w=\sum_\ell w_\ell$.
	Since $\rho$ is fixed, the number of terms is bounded independently of
	$\lambda,\varepsilon$. Equations \eqref{eq:local-inputs} and \eqref{eq:trace} give
	\[
	\begin{gathered}
		\lambda_\ell=\lambda\rho,\qquad
		K_{g_\ell}(z)=\varepsilon\rho^2F(p_\ell+\rho z),\qquad
		0<\varepsilon\rho^2\le\varepsilon_0,\\
		\|f_{j,\ell}\|_{L^2(dV_{g_\ell})}
		\le\|V_{\lambda,j}\|_{L^2(dV_g)}\le C.
	\end{gathered}
	\]
	All defining geodesics lie in the curvature neighborhood.
	After a fixed finite partition, each kernel is either nonstationary or
	has the form in Lemma~\ref{lem:kernel-phase} in normal coordinates at a
	spatial critical point. On the latter pieces, apply
	Lemmas~\ref{lem:kernel-phase} and~\ref{lem:trilinear} at fixed $t_1,t_2,t_3$,
	then integrate in these variables by \eqref{eq:iterated-input-bound}.
	Using $\|f\|_1\le C\|f\|_2$ on the fixed intervals for the kernel remainder
	and \eqref{eq:nonstationary} for the other pieces gives
	\[
	\begin{aligned}
		\left|\int w_\ell(p_\ell+\rho x)
		\prod_{j=1}^3(Q_{j,\ell}f_{j,\ell})(x)\dd V_{g_\ell}(x)\right|\le C\left[(1+\lambda\varepsilon\rho^3)^{-1/6}
		+\lambda_\ell^{-1/2}\right],
	\end{aligned}
	\]
	where $\lambda_\ell(\varepsilon\rho^2)=\lambda\varepsilon\rho^3$.
	Summing by \eqref{eq:localization-bound}, with $\rho$ fixed, yields
	\[
	\begin{aligned}
		\left|\int w\prod_{j=1}^3E_{\lambda,j}f_j\dd V_g\right|
		\le C_\rho\left[(1+\lambda\varepsilon\rho^3)^{-1/6}
		+\lambda^{-1/2}\right]\le C_\rho(1+\lambda\varepsilon)^{-1/6}.
	\end{aligned}
	\]
\end{proof}

	\subsection{Finite-order vanishing of curvature}\label{sec:finite}
	\begin{lemma}\label{lem:finite-order}
		Let $k\ge0$ be an integer. On a fixed neighborhood of $\supp w$, suppose
		\begin{equation}\label{eq:uniform-class}
			K_g=\varepsilon F,\qquad 0<\varepsilon\le1,\qquad
			\|F\|_{C^N}\le C_N,\qquad
			\max_{|\alpha|\le k}|\partial^\alpha F(x)|\ge c>0.
		\end{equation}
		For \eqref{eq:distance-operator} and weights with uniform smooth bounds,
		\begin{equation}\label{eq:all-operators}
			\left|\int w\prod_{j=1}^3 E_{\lambda,j}f_j\dd V_g\right|
			\le C(1+\lambda\varepsilon)^{-1/[2(k+3)]}\prod_{j=1}^3\|f_j\|_2.
		\end{equation}
	\end{lemma}
	\begin{proof}
		The case $k=0$ is Lemma~\ref{lem:nonzero-operator}. Let $k\ge1$ and
		assume the conclusion for $k-1$. By \eqref{eq:nonstationary},
		\eqref{eq:trace}, and \eqref{eq:bilinear}, it suffices to consider
		the second case of \eqref{eq:angle-alternatives}, with
		$\lambda\varepsilon$ sufficiently large and $\|f_j\|_2=1$.
		Write $x=(x_1,x_2)$ throughout.

		After a fixed finite partition, either
		$\max_{|\alpha|\le k-1}|\partial^\alpha F|\ge c'$, in which case
		the induction hypothesis gives
		\[
		\left|\int w\prod_{j=1}^3E_{\lambda,j}f_j\dd V_g\right|
		\le C(1+\lambda\varepsilon)^{-1/[2(k+2)]}
		\le C(1+\lambda\varepsilon)^{-1/[2(k+3)]},
		\]
		or, interchanging $x_1,x_2$ if necessary, a fixed multi-index $\beta$ satisfies
		\begin{equation}\label{eq:curvature-derivative}
			h_F=\partial^\beta F,\qquad |\beta|=k-1,\qquad
			|\partial_{x_2}h_F|\ge c'>0.
		\end{equation}
		By \eqref{eq:angle-alternatives},
		$\max_{1\le j\le3}|\langle\nabla_xh_F,\nabla_x\phi_j\rangle_g|
		\ge c|\nabla_xh_F|_g\ge c$.
		Refine the fixed partition and relabel so that
		$|\langle\nabla_xh_F,\nabla_x\phi_3\rangle_g|\ge c$ on a fixed
		neighborhood of $\supp w$, and choose $\chi_1$ in
		Lemma~\ref{lem:radial} supported there.
		Use coordinates $(x_1,t)$ with $t=h_F(x)$ and write $x_2=x_2(x_1,t)$.
		Differentiating $h_F(x_1,x_2(x_1,t))=t$ at fixed $t$ gives
		\[
		\partial_{x_1}h_F+\partial_{x_2}h_F
		\left.\frac{\partial x_2}{\partial x_1}\right|_t=0,
		\qquad
		\left.\partial_{x_1}\right|_t
		=\partial_{x_1}
		-\frac{\partial_{x_1}h_F}{\partial_{x_2}h_F}\partial_{x_2}.
		\]
		By \eqref{eq:direction-bounds},
		\[
		\begin{gathered}
			\left|\left.\partial_{x_1}\right|_t\partial_{s_3}\phi_3\right|
			=\frac{\sqrt{\det g}}{|\partial_{x_2}h_F|}
			|\partial_{s_3}\nabla_x\phi_3|_g
			|\langle\nabla_xh_F,\nabla_x\phi_3\rangle_g|\ge c,\quad\quad
			dV_g=\frac{\sqrt{\det g}}{|\partial_{x_2}h_F|}\dd x_1\dd t.
		\end{gathered}
		\]
		Applying \eqref{eq:nondegenerate-L2} in $x_1,s_3$ for fixed $t$ gives,
		with the functions written in these coordinates,
		\[
		\sup_t\left(
		\|(E_{\lambda,3}f_3)(\cdot,t)\|_{L^2(dx_1)}
		+\|V_{\lambda,3}(\cdot,t)\|_{L^2(dx_1)}\right)
		\le C\lambda^{1/2}\lambda^{-1/2}=C.
		\]
		For $r>0$, integration of the squared bounds over $|t|<r$ yields
		\begin{equation}\label{eq:derivative-band}
			\|E_{\lambda,3}f_3\|_{L^2(\{|h_F|<r\})}
			+\|V_{\lambda,3}\|_{L^2(\{|h_F|<r\})}\le Cr^{1/2}.
		\end{equation}

		Set
		\begin{equation}\label{eq:delta}
			\delta=(\lambda\varepsilon)^{-1/(k+3)}.
		\end{equation}
		Choose an even $\chi\in C_c^\infty((-2,2))$, equal to one on $[-1,1]$,
		with $0\le\chi\le1$. Define
		\[
		b_r=w[\chi(h_F/(2r))-\chi(h_F/r)],\qquad
		w=w\chi(h_F/\delta)+\sum_{q\ge0}b_{2^q\delta}.
		\]
		The sum is finite, and $\supp b_r\subset\{r\le|h_F|\le4r\}$.
		By \eqref{eq:bilinear} and \eqref{eq:derivative-band},
		\begin{equation}\label{eq:core}
			\begin{aligned}
				\left|\int w\chi(h_F/\delta)
				\prod_{j=1}^3E_{\lambda,j}f_j\dd V_g\right|
				\le C\|E_{\lambda,1}f_1E_{\lambda,2}f_2\|_2
				\|E_{\lambda,3}f_3\|_{L^2(\{|h_F|<2\delta\})}&\le C\delta^{1/2}.
			\end{aligned}
		\end{equation}

		It remains to estimate $b_r$ for $r=2^q\delta$.
		Fix $a_*>0$ sufficiently small and use Section~\ref{sec:local-inputs}
		at scale $\rho=a_*r\le\min(\rho_0,1)$:
		\[
		b_r=\sum_\ell w_\ell,\qquad
		\supp w_\ell\subset B(p_\ell,c_0\rho),\qquad
		|\partial^\alpha w_\ell|\le C_\alpha\rho^{-|\alpha|}.
		\]
		The balls $B(p_\ell,4\rho)$ have bounded overlap. Since
		$p_\ell\in\supp b_r$ and
		$|h_F(x)-h_F(p_\ell)|\le4\rho\|\nabla_xh_F\|_\infty\le Ca_*r$
		for $x\in B(p_\ell,4\rho)$,
		\[
		B(p_\ell,4\rho)\subset\{r/2<|h_F|<5r\}.
		\]
		To apply the induction hypothesis on these balls, set
		\begin{equation}\label{eq:curvature-rescaling}
			\nu_\ell=\max_{|\alpha|\le k-1}\rho^{|\alpha|}|\partial^\alpha F(p_\ell)|,
			\qquad F_\ell(z)=\frac{F(p_\ell+\rho z)}{\nu_\ell},\qquad
			\varepsilon_\ell=\varepsilon\rho^2\nu_\ell.
		\end{equation}
		Then
		\[
		\rho^{k-1}r\le\rho^{k-1}|h_F(p_\ell)|\le\nu_\ell\le C,
		\qquad \frac{\rho^k}{\nu_\ell}\le a_*,\qquad
		0<\varepsilon_\ell\le C\rho^2\le1.
		\]
		By differentiation,
		\begin{equation}\label{eq:child-bounds}
			\max_{|\alpha|\le k-1}|\partial^\alpha F_\ell(0)|=1,\qquad
			\max_{|\alpha|=k}\|\partial^\alpha F_\ell\|_{C^N(B(0,4))}
			\le C_N\frac{\rho^k}{\nu_\ell}\le C_Na_*.
		\end{equation}
		Taylor's formula and \eqref{eq:child-bounds} give, for
		$z\in B(0,4)$ and $|\alpha|\le k-1$,
		\[
		\begin{aligned}
		\partial^\alpha F_\ell(0)
		&=\sum_{|\gamma|\le k-1-|\alpha|}
		\frac{(-z)^\gamma}{\gamma!}\partial^{\alpha+\gamma}F_\ell(z)+O(a_*),\\
		\partial^\alpha F_\ell(z)
		&=\sum_{|\gamma|\le k-1-|\alpha|}
		\frac{z^\gamma}{\gamma!}\partial^{\alpha+\gamma}F_\ell(0)+O(a_*).
		\end{aligned}
		\]
		Since $\max_{|\alpha|\le k-1}|\partial^\alpha F_\ell(0)|=1$, these imply
		\[
		1\le C_k\max_{|\alpha|\le k-1}|\partial^\alpha F_\ell(z)|+C_ka_*,
		\qquad
		\max_{|\alpha|\le k-1}|\partial^\alpha F_\ell(z)|\le C_k.
		\]
		Choose $a_*$ so that $C_ka_*\le1/2$. Together with
		\eqref{eq:child-bounds} for derivatives of order at least $k$, this gives
		\[
		\|F_\ell\|_{C^N(B(0,4))}\le C_N,\qquad
		\inf_{z\in B(0,4)}\max_{|\alpha|\le k-1}
		|\partial^\alpha F_\ell(z)|\ge\frac1{2C_k}>0.
		\]
		For $g_\ell,\lambda_\ell$ from \eqref{eq:local-inputs},
		\begin{equation}\label{eq:effective-frequency}
			K_{g_\ell}=\varepsilon_\ell F_\ell,\qquad
			\lambda_\ell\varepsilon_\ell=\lambda\varepsilon\rho^3\nu_\ell
			\ge a_*^{k+2}\lambda\varepsilon r^{k+3}.
		\end{equation}
		For large $\lambda\varepsilon$,
		$\lambda^{1/2}\rho\ge a_*\lambda^{1/2}\delta
		=a_*\varepsilon^{-1/2}(\lambda\varepsilon)^{(k+1)/[2(k+3)]}\ge1$.
		The preceding bounds verify \eqref{eq:uniform-class} with $k-1$ in place of $k$.
		Apply the induction hypothesis to the inner $s_j$ integrals in $Q_{j,\ell}$,
		then integrate $t_1,t_2,t_3$ by \eqref{eq:iterated-input-bound} to obtain
		\[
		\begin{aligned}
			&\left|\int w_\ell(p_\ell+\rho x)
			\prod_{j=1}^3(Q_{j,\ell}f_{j,\ell})(x)\dd V_{g_\ell}(x)\right|\\
			&\quad\le C(1+\lambda_\ell\varepsilon_\ell)^{-1/[2(k+2)]}
			\prod_{j=1}^3m_{j,\ell}
			\le C_{a_*}(\lambda\varepsilon r^{k+3})^{-1/[2(k+2)]}
			\prod_{j=1}^3m_{j,\ell}.
		\end{aligned}
		\]
		By bounded overlap, \eqref{eq:derivative-band}, and \eqref{eq:pair-input-sum},
		\[
		\begin{aligned}
			\sum_\ell m_{3,\ell}^2
			&\le C\|V_{\lambda,3}\|_{L^2(\{|h_F|<5r\})}^2\le Cr,\\
			\sum_\ell\prod_{j=1}^3m_{j,\ell}
			&\le\left(\sum_\ell m_{1,\ell}^2m_{2,\ell}^2\right)^{1/2}
			\left(\sum_\ell m_{3,\ell}^2\right)^{1/2}
			\le C\rho r^{1/2}.
		\end{aligned}
		\]
		Equation \eqref{eq:localization-bound} now gives
		\begin{equation}\label{eq:shell}
			\begin{aligned}
				\left|\int b_r\prod_{j=1}^3E_{\lambda,j}f_j\dd V_g\right|
				&\le C_{a_*}\rho^{-1}(\lambda\varepsilon r^{k+3})^{-1/[2(k+2)]}
				\sum_\ell\prod_{j=1}^3m_{j,\ell}+\frac{C}{\lambda^{1/2}\rho}\\
				&\le C_{a_*}\left[
				r^{1/2}(\lambda\varepsilon r^{k+3})^{-1/[2(k+2)]}
				+\lambda^{-1/2}r^{-1}\right].
			\end{aligned}
		\end{equation}
		Since $r=2^q\delta$ and $\lambda\varepsilon\delta^{k+3}=1$,
		\begin{equation}\label{eq:product-error}
			\begin{gathered}
				r^{1/2}(\lambda\varepsilon r^{k+3})^{-1/[2(k+2)]}
				=\delta^{1/2}2^{-q/[2(k+2)]},\\
				\lambda^{-1/2}r^{-1}
				=\varepsilon^{1/2}\delta^{(k+1)/2}2^{-q}
				\le\delta^{1/2}2^{-q}.
			\end{gathered}
		\end{equation}
		Summing \eqref{eq:core} and \eqref{eq:shell}, we obtain
		\[
		\left|\int w\prod_{j=1}^3E_{\lambda,j}f_j\dd V_g\right|
		\le C_k\delta^{1/2}
		\left(1+\sum_{q\ge0}2^{-q/[2(k+2)]}\right)
		\le C_k(\lambda\varepsilon)^{-1/[2(k+3)]}.
		\]
		Homogeneity proves \eqref{eq:all-operators}.
	\end{proof}
	
	\subsection{Proof of Theorem~\ref{thm:upper}}
	In the fixed coordinates, including the zeroth derivative,
	\begin{equation}\label{eq:covariant-jet-comparison}
		\sum_{j=0}^m|\nabla^jK_g|_g
		\approx\max_{|\alpha|\le m}|\partial^\alpha K_g|.
	\end{equation}
	Each covariant derivative is its coordinate derivative plus smooth
	linear combinations of lower derivatives, so the comparison follows
	by induction.
	\begin{proof}
		Part~\textup{(i)} is Lemma~\ref{lem:small-weight} with $b=w$.
		For \textup{(iii)}, \eqref{eq:covariant-jet-comparison} and continuity give
		$\max_{|\alpha|\le m}|\partial^\alpha K_g|\ge c>0$ on a fixed
		neighborhood of $\supp w$. Apply Lemma~\ref{lem:finite-order} with
		$k=m$ and $\varepsilon=1$, then integrate the parameters using
		\eqref{eq:iterated-input-bound} and \eqref{eq:sogge-reduction}. Interior elliptic
		estimates cover bounded frequencies.
		In particular, for fixed $b\in C_c^\infty(U)$ with $K_g\ne0$ on
		$\supp b$, the case $m=0$ gives
		\begin{equation}\label{eq:nonzero-bound}
			\left|\int_Ub u_1u_2u_3\dd V_g\right|\le C_b\lambda^{-1/6}.
		\end{equation}
		
		For \textup{(ii)}, choose $\chi_1\in C_c^\infty(\R)$ with $0\le\chi_1\le1$,
		$\chi_1=1$ on $[-1,1]$, and $\supp\chi_1\subset[-2,2]$.
		For each fixed $t>0$, split $w=w[1-\chi_1(K_g/t)]+w\chi_1(K_g/t)$.
		Equations \eqref{eq:small-weight} and \eqref{eq:nonzero-bound} imply
		\[
		\left|\int_Uw u_1u_2u_3\dd V_g\right|\le
		C_t\lambda^{-1/6}+C\|w\chi_1(K_g/t)\|_{L^4(U)},
		\]
		where $C$ is independent of $t$. Let $\lambda\to\infty$
		and then $t\downarrow0$. Uniformly over these solutions,
		\[
		\limsup_{\lambda\to\infty} \left|\int_Uw u_1u_2u_3\dd V_g\right|
		\le C\left(\int_{\{K_g=0\}}|w|^4\dd V_g\right)^{1/4}=0.
		\]
		Take $u_1=u_2=u_3=e_\lambda$ and sum the initial finite partition.
	\end{proof}
	
	For use in the proof of Theorem~\ref{thm:global}, we prove the following
	quantitative form of the case $m=0$. For $w\in C_c^\infty(U)$ with
	$K_g\ne0$ on $\supp w$,
	\begin{equation}\label{eq:weight-c2}
		\left|\int_Uw u_1u_2u_3\dd V_g\right|
		\le C_{g,\supp w,U}\lambda^{-1/6}
		\left(\sum_{j=0}^2\sup_U|\nabla^jw|_g\right)
		\prod_{j=1}^3\|u_j\|_{L^2(U)}.
	\end{equation}
	The constant depends on $w$ only through its support and is uniform when
	the supports stay in a fixed compact subset of $\{K_g\ne0\}\cap U$.
	By a fixed partition and homogeneity, work in one coordinate chart with
	$\|u_j\|_{L^2(U)}=1$. Choose a fixed $\chi\in C_c^\infty(U)$ equal to one
	near $\supp w$ and supported where $K_g\ne0$.
	Fix the rescaling in Lemma~\ref{lem:nonzero-operator}, then choose $c>0$
	sufficiently small compared with the resulting curvature parameter.
	The lower bounds in \eqref{eq:spatial-hessian}, \eqref{eq:geometric-mixed}, and
	\eqref{eq:curvature-nondegeneracy} persist after adding
	$(\xi/\lambda)\cdot x$ to the phase in $x$ when $|\xi|\le c\lambda$.
	Absorbing $e^{i\xi\cdot x}$ into the phase and applying that proof gives
	the first bound below. For the second, \eqref{eq:sogge} and
	Cauchy--Schwarz give $\|u_1\|_{L^\infty(\supp\chi)}\le C\lambda^{1/2}$
	and $\|u_2u_3\|_{L^1(U)}\le1$:
	\[
	\left|\int\chi(x)e^{i\xi\cdot x}\prod_{j=1}^3u_j(x)\dd V_g(x)\right|
	\le C\begin{cases}
		\lambda^{-1/6},&|\xi|\le c\lambda,\\
		\lambda^{1/2},&|\xi|>c\lambda.
	\end{cases}
	\]
	Extend $w$ by zero in these coordinates. Fourier inversion and
	Cauchy--Schwarz yield
	\[
	\begin{aligned}
		\left|\int w\prod_{j=1}^3u_j\dd V_g\right|
		&\le C\lambda^{-1/6}\|\widehat w\|_1
		+C\lambda^{1/2}\int_{|\xi|>c\lambda}|\widehat w(\xi)|\dd\xi\\
		&\le C(\lambda^{-1/6}+\lambda^{-1/2})\|w\|_{H^2(\R^2)}
		\le C\lambda^{-1/6}\sum_{j=0}^2\sup_U|\nabla^jw|_g.
	\end{aligned}
	\]
	
	\section{Global estimates on hyperbolic surfaces}\label{sec:global}
	
	To prove Theorem~\ref{thm:global}, we use the thick--thin decomposition
	to combine local estimates with Fourier--Bessel analysis in the cusps.
	
	Throughout this section, set
	\[
	-\Delta_g e_\lambda=\lambda^2e_\lambda,\qquad
	\mu=(\lambda^2-\tfrac14)^{1/2}\approx\lambda.
	\]
	All $L^p$ norms on the surface use $dV_g$. We use the upper half-plane
	$\HH=\{z=x+iy:y>0\}$ with metric $g=y^{-2}(dx^2+dy^2)$.
	
	\subsection{Local estimates and the thick--thin decomposition}
	Fix sufficiently small coordinate balls $U_0\Subset U\Subset\HH$.
	For $-\Delta_g u=\lambda^2u$ on $U$, Sogge's estimate
	\cite[Theorem~2.2(ii)]{Sogge1988}, by duality and \eqref{eq:sogge}, gives
	\begin{equation}\label{eq:global-absolute-local}
		\|u\|_{L^3(U_0)}^3\le C\lambda^{1/4}\|u\|_{L^2(U)}^3.
	\end{equation}
	For $b\in C_c^\infty(U_0)$, the bound for
	$\int b u^3\dd V_g$ is \eqref{eq:weight-c2} with $u_1=u_2=u_3=u$. Both local constants are
	uniform for weights supported in $U_0$ and unchanged by hyperbolic
	isometries.
	
	We may assume that $M$ is connected. Let $\operatorname{inj}_M$ denote
	its injectivity radius. For $r>0$, set
	$M_{<r}=\{p:\operatorname{inj}_M(p)<r\}$ and $M_{\ge r}=M\setminus M_{<r}$.
	By the thick--thin decomposition
	\cite[Sections~D.2--D.3]{BenedettiPetronio} and
	$\operatorname{sys}(M)\ge s_0$, fix $\delta=\delta(s_0)>0$ sufficiently small
	that every component of $M_{<2\delta}$ is a cusp neighborhood. Write
	$M_{<\delta}=\bigsqcup_\ell Y_\ell$.
	
	For each $Y_\ell$, conjugate the cusp stabilizer to
	$\langle z\mapsto z+1\rangle$. Then
	\[
	\cosh d_g(z,z+1)=1+\frac1{2y^2},\qquad
	\operatorname{inj}_M(x,y)=\tfrac12d_g(z,z+1)=\operatorname{arsinh}\frac1{2y}.
	\]
	We therefore identify each component with
	\begin{equation}\label{eq:global-cusps}
		Y_\ell=(\R/\mathbb Z)\times(y_0,\infty),\qquad
		y_0=(2\sinh\delta)^{-1}\ge1,\qquad
		dV_g=y^{-2}\dd x\dd y.
	\end{equation}
	The family $\{Y_\ell\}$ is countable and locally finite by second countability
	and $\overline{M_{<\delta}}\subset M_{<2\delta}$, each component
	of the latter containing exactly one $Y_\ell$.
	
	Throughout this section, fix a nondecreasing $\chi\in C^\infty((0,\infty))$ with
	$0\le\chi\le1$, $\chi=0$ on $(0,1]$, and $\chi=1$ on $[2,\infty)$.
	For $H>y_0$, extend each $\chi(y/H)$ by zero outside $Y_\ell$ and set
	\begin{equation}\label{eq:global-cutoff}
		w_H=1-\sum_\ell\mathbf1_{Y_\ell}\chi(y/H).
	\end{equation}
	The summands are smooth and locally finite. The isometry $z\mapsto z/H$
	gives $\sup_M|\nabla^j w_H|_g\le C_j$ for $0\le j\le2$, uniformly in $H$.
	
	\subsection{Fourier expansions in the cusps}
	Set $W_\lambda(t)=e^{\pi\mu/2}K_{i\mu}(t)$, where $K_\nu$ is the modified
	Bessel function of the second kind. The function $W_\lambda$ is real for $t>0$,
	and all bounds below are uniform for $\mu\ge1$.
	Here $\Gamma$ denotes the gamma function.
	\begin{lemma}\label{lem:bessel}
		For $0<t_0<t_1\le\mu/4$,
		\begin{equation}\label{eq:global-mass}
			\int_{t_0}^{t_1}W_\lambda(t)^2\frac{\dd t}{t}
			\ge c\mu^{-1}\log(t_1/t_0)-C\mu^{-2}.
		\end{equation}
		For $t\ge t_0\ge2\mu$,
		\begin{equation}\label{eq:global-decay}
			0<W_\lambda(t)\le e^{-(t-t_0)/2}W_\lambda(t_0).
		\end{equation}
	\end{lemma}
	\begin{proof}
		Let $D_t=t\partial_t$. The expansions at $t=0$ and the gamma-function identity
		\cite[Eqs.~10.25.2, 10.27.4, and~5.4.3]{NIST} give
		\[
		\lim_{t\downarrow0}\bigl((D_t W_\lambda)^2+\mu^2W_\lambda^2\bigr)
		=\mu^2e^{\pi\mu}|\Gamma(i\mu)|^2
		=\frac{2\pi\mu}{1-e^{-2\pi\mu}}\approx\mu.
		\]
		The modified Bessel equation \cite[Eq.~10.45.1]{NIST},
		$D_t^2W_\lambda+(\mu^2-t^2)W_\lambda=0$, implies
		\[
		\begin{aligned}
			D_t \bigl((D_t W_\lambda)^2+(\mu^2-t^2)W_\lambda^2\bigr)&=-2t^2W_\lambda^2,\\
			D_t \left(W_\lambda^2+\frac{(D_t W_\lambda)^2}{\mu^2-t^2}\right)
			&=\frac{2t^2(D_t W_\lambda)^2}{(\mu^2-t^2)^2}.
		\end{aligned}
		\]
		Thus, on $0<t\le\mu/4$,
		\[
		c\mu\le (D_t W_\lambda)^2+(\mu^2-t^2)W_\lambda^2\le C\mu,
		\qquad |W_\lambda\,D_t W_\lambda|\le C,
		\]
		and therefore
		\begin{align*}
			2\mu^2\int_{t_0}^{t_1}W_\lambda^2\frac{\dd t}{t}
			&\ge2\int_{t_0}^{t_1}(\mu^2-t^2)W_\lambda^2\frac{\dd t}{t}\\
			&=\int_{t_0}^{t_1}\bigl((D_t W_\lambda)^2+(\mu^2-t^2)W_\lambda^2\bigr)
			\frac{\dd t}{t}-[W_\lambda\,D_t W_\lambda]_{t_0}^{t_1}\\
			&\ge c\mu\log(t_1/t_0)-C,
		\end{align*}
		proving \eqref{eq:global-mass}.
		
		For fixed $\mu$, the asymptotics as $t\to\infty$
		\cite[Eqs.~10.40.2 and~10.40.4]{NIST} give
		$W_\lambda>0$ near infinity and $W_\lambda,tW_\lambda'\to0$. If $z\ge2\mu$ were its last
		zero, then
		\[
		(tW_\lambda')'=(t-\mu^2/t)W_\lambda>0\quad(t>z)
		\quad\Longrightarrow\quad W_\lambda'<0,
		\]
		contradicting $W_\lambda(z)=0<W_\lambda(t)$. Hence, for $t\ge2\mu$,
		\[
		W_\lambda>0,\qquad W_\lambda'<0,\qquad
		W_\lambda''=(1-\mu^2/t^2)W_\lambda-W_\lambda'/t\ge\tfrac34W_\lambda.
		\]
		Consequently
		\[
		\left(e^{-t/2}(W_\lambda'+W_\lambda/2)\right)'\ge0,\qquad
		e^{-t/2}(W_\lambda'+W_\lambda/2)\longrightarrow0,
		\]
		so $W_\lambda'+W_\lambda/2\le0$, proving \eqref{eq:global-decay}.
	\end{proof}
	
	\begin{lemma}\label{lem:cusp-fourier}
		For any $y_0>0$, let $u\in L^2(Y)$ solve $-\Delta_g u=\lambda^2u$ on
		$Y=(\R/\mathbb Z)\times(y_0,\infty)$ with the hyperbolic metric above.
		There are constants $c_n\in\mathbb C$, $n\ne0$, such that
		\begin{equation}\label{eq:global-fourier}
			\begin{split}
				u(x,y)&=\sqrt y\sum_{n\ne0}c_nW_\lambda(2\pi|n|y)e^{2\pi inx},\\
				\|u\|_{L^2(Y)}^2
				&=\sum_{n\ne0}|c_n|^2
				\int_{y_0}^\infty W_\lambda(2\pi|n|y)^2\frac{\dd y}{y}.
			\end{split}
		\end{equation}
		For an integer $N\ge1$, let $u_{\ge N}$ be the sum in
		\eqref{eq:global-fourier} restricted to $|n|\ge N$.
		If $H\ge4\max\{y_0,1\}$ and $NH\ge4\mu$, then
		\begin{equation}\label{eq:global-high-bound}
			|u_{\ge N}(x,y)|\le Ce^{-cNy}\|u\|_{L^2(Y)},\qquad y\ge H.
		\end{equation}
	\end{lemma}
	\begin{proof}
		Taking Fourier coefficients in $x$ as in
		\cite[Proposition~1.5]{Iwaniec} gives $y^{1/2\pm i\mu}$ for $n=0$ and
		$\sqrt y\,I_{i\mu}(2\pi|n|y)$, $\sqrt y\,K_{i\mu}(2\pi|n|y)$ for $n\ne0$,
		where $I_\nu$ is the modified Bessel function of the first kind.
		Since $\mu>0$, no nonzero linear combination of the two solutions for
		$n=0$ belongs to $L^2((y_0,\infty),y^{-2}\dd y)$. For $n\ne0$, the
		exponentially growing $I_{i\mu}$ term is excluded by square integrability
		\cite[Eqs.~10.40.1--10.40.2]{NIST}. The remaining expansion and Parseval's identity give \eqref{eq:global-fourier}.
		For $|n|\ge N$ and $H/4\le y\le H/2$, we have
		$2\pi|n|y\ge(\pi/2)NH>2\mu$. Monotonicity gives
		\[
		\sum_{|n|\ge N}|c_n|^2W_\lambda(\pi|n|H)^2
		\le(\log2)^{-1}\|u\|_{L^2(Y)}^2.
		\]
		For $y\ge H$, \eqref{eq:global-decay} gives
		$W_\lambda(2\pi|n|y)\le e^{-\pi|n|y/2}W_\lambda(\pi|n|H)$.
		Thus Cauchy--Schwarz yields
		\[
		|u_{\ge N}(x,y)|
		\le C\sqrt y\,\|u\|_{L^2(Y)}
		\left(\sum_{|n|\ge N}e^{-\pi|n|y}\right)^{1/2}
		\le Ce^{-\pi Ny/4}\|u\|_{L^2(Y)}.\qedhere
		\]
	\end{proof}
	
	\subsection{A weighted integral of three Bessel functions}
	Integration in the horocycle variable selects Fourier modes with
	$n_1+n_2+n_3=0$, leading to the following Bessel integral.
	
	\begin{lemma}\label{lem:bessel-triple}
		For $a,b>0$, $\mu\ge1$, and $\omega\in C^2((0,\infty))$,
		\begin{equation}\label{eq:global-triple-bound}
			\begin{aligned}
				&\left|\int_0^\infty\omega(y)y^{-1/2}
				W_\lambda(ay)W_\lambda(by)W_\lambda((a+b)y)\dd y\right|\\
				&\qquad\le C\mu^{-3/2}(a+b)^{-1/2}
				\sum_{j=0}^2\|(y\partial_y)^j\omega\|_\infty,
			\end{aligned}
		\end{equation}
		provided the sum on the right is finite. The constant is independent of
		$a/b$.
	\end{lemma}
	\begin{proof}
		Normalize the sum on the right to be at most one.
		The integral formulas \cite[(10.32.7) and (10.32.9)]{NIST} give
		\begin{equation}\label{eq:bessel-oscillatory}
			W_\lambda(t)=\frac12\int_\R e^{i(\mu s-t\sinh s)}\dd s,\qquad
			|W_\lambda(t)|\le e^{\pi\mu/2}\int_0^\infty e^{-t\cosh s}\dd s.
		\end{equation}
		The first integral is understood with smooth cutoffs tending to one.
		Choose $\chi_1\in C_c^\infty((1/2,2))$ equal to one on $[3/4,5/4]$.
		On the support of $1-\chi_1(z\cosh s)$,
		$|1-z\cosh s|\ge1/4$.
		Three integrations by parts, splitting at $s=0$, give uniformly for $z>0$
		\begin{equation}\label{eq:bessel-cutoff}
			W_\lambda(\mu z)=\frac12\int_\R e^{i\mu(s-z\sinh s)}
			\chi_1(z\cosh s)\dd s+O(\mu^{-3}).
		\end{equation}
		The set $\{s:\chi_1(z\cosh s)\ne0\}$ has uniformly bounded length, so
		\begin{equation}\label{eq:bessel-global-sup}
			\sup_{t>0}|W_\lambda(t)|\le C.
		\end{equation}
		Set $\alpha=a/(a+b)$, $(\alpha_1,\alpha_2,\alpha_3)=(\alpha,1-\alpha,1)$,
		and $y=\mu r/(a+b)$. It suffices to prove
		\begin{equation}\label{eq:bessel-scaled-weight}
			\left|\int_0^\infty r^{-1/2}\omega\left(\frac{\mu r}{a+b}\right)
			\prod_{j=1}^3W_\lambda(\mu\alpha_jr)\dd r\right|\le C\mu^{-2}.
		\end{equation}
		For $r\ge2$, \eqref{eq:bessel-oscillatory} and $\cosh s\ge1+s^2/2$ give
		\[
		|W_\lambda(\mu r)|\le C(\mu r)^{-1/2}e^{-\mu(r-\pi/2)}
		\le C(\mu r)^{-1/2}e^{-c\mu r}.
		\]
		Together with \eqref{eq:bessel-global-sup}, this makes the part $r\ge2$
		exponentially small.
		
		On $0<r<2$, substitute \eqref{eq:bessel-cutoff} and change variables by
		\[
		p_j=r\alpha_j\sinh s_j,\qquad
		q_j=(p_j^2+r^2\alpha_j^2)^{1/2},\qquad
		\dd s_j=q_j^{-1}\dd p_j.
		\]
		Up to $O(\mu^{-3})$, the resulting integral is
		\[
		\frac18\int_0^2\int_{\R^3} e^{i\mu\Phi(r,p)}r^{-1/2}
		\omega\left(\frac{\mu r}{a+b}\right)
		\prod_{j=1}^3\frac{\chi_1(q_j)}{q_j}\dd p\dd r,\qquad
		\Phi(r,p)=\sum_{j=1}^3\left[\operatorname{arsinh}
		\left(\frac{p_j}{r\alpha_j}\right)-p_j\right].
		\]
		On the amplitude support, $|p_j|\le2$ and $1/2\le q_j\le2$. Let
		$D=r\partial_r$. Since
		\[
		\begin{gathered}
			\partial_{p_j}q_j=\frac{p_j}{q_j},\qquad
			Dq_j=\frac{r^2\alpha_j^2}{q_j},\\
			\partial_{p_j}\Phi=q_j^{-1}-1,\qquad
			D\Phi=-\sum_{j=1}^3\frac{p_j}{q_j},
		\end{gathered}
		\]
		for $m\ge0$ and $\beta\in\mathbb N_0^3$ we have
		\[
		\left|\partial_p^\beta D^m
		\prod_{j=1}^3\frac{\chi_1(q_j)}{q_j}\right|\le C_{\beta,m},\qquad
		|\partial_p^\beta D^m\Phi|\le C_{\beta,m}\quad(m+|\beta|\ge1),
		\]
		uniformly for $0<\alpha<1$.
		
		Choose $0\le\chi_0\le1$ smooth, equal to one for $r\le1/8$ and zero for
		$r\ge1/4$. On its intersection with the amplitude support,
		\[
		\frac{|p_j|}{q_j}=\sqrt{1-\frac{r^2\alpha_j^2}{q_j^2}}
		\ge\frac{\sqrt3}{2},\qquad |D\Phi|\ge\sqrt3-1.
		\]
		Set
		\[
		a_0(r,p)=\chi_0(r)\omega\left(\frac{\mu r}{a+b}\right)
		\prod_{j=1}^3\frac{\chi_1(q_j)}{q_j},\qquad
		a_{j+1}=-\left(D+\tfrac12\right)\frac{a_j}{D\Phi}\quad(j=0,1).
		\]
		The assumption $\sum_{j=0}^2\|(y\partial_y)^j\omega\|_\infty\le1$ and
		\[
		D^j\!\left[\omega\left(\frac{\mu r}{a+b}\right)\right]
		=\bigl[(y\partial_y)^j\omega\bigr]\left(\frac{\mu r}{a+b}\right),
		\qquad 0\le j\le2,
		\]
		give $|a_j|\le C$ for $0\le j\le2$. The boundary terms satisfy
		\[
		\sup_p\left|r^{1/2}e^{i\mu\Phi}\frac{a_j}{D\Phi}\right|
		\le Cr^{1/2}\longrightarrow0\quad(r\downarrow0),\qquad j=0,1.
		\]
		Thus two integrations by parts, using
		$De^{i\mu\Phi}=i\mu(D\Phi)e^{i\mu\Phi}$, yield
		\[
		\left|\int_0^\infty\int_{\R^3}e^{i\mu\Phi}r^{-1/2}a_0\dd p\dd r\right|
		=\mu^{-2}\left|\int_0^\infty\int_{\R^3}
		e^{i\mu\Phi}r^{-1/2}a_2\dd p\dd r\right|
		\le C\mu^{-2}.
		\]
		
		For $r>1/8$, the identity
		$\partial_r^m=r^{-m}\prod_{j=0}^{m-1}(D-j)$ for $m\ge1$ also gives
		\[
		|\partial_{(r,p)}^\gamma\Phi|\le C_\gamma,\qquad |\gamma|\ge1.
		\]
		The critical point equations $q_j=1$ and $\sum_{j=1}^3p_j=0$ have exactly two
		solutions $(r_\alpha,\pm p_\alpha)$, where
		\[
		r_\alpha^2=\frac3{4(1-\alpha+\alpha^2)},\qquad
		p_\alpha=\frac{r_\alpha}{\sqrt3}(2-\alpha,-1-\alpha,2\alpha-1).
		\]
		At either solution,
		\begin{equation}\label{eq:bessel-joint-hessian}
			\sum_{j=1}^3p_j^2=\tfrac32,\qquad
			\det\nabla^2\Phi=-\frac1{r_\alpha^2}\sum_{1\le i<j\le3}p_ip_j
			=\frac3{4r_\alpha^2}\ge\tfrac34.
		\end{equation}
		By \eqref{eq:bessel-joint-hessian} and the uniform third-derivative bounds,
		choose a fixed $\eta\in C_c^\infty(\R^4)$, equal to one near zero and
		supported in a sufficiently small ball, so that
		\[
		\begin{gathered}
			\eta_\pm(r,p)=\eta(r-r_\alpha,p\mp p_\alpha),\qquad
			A_\pm=\eta_\pm\prod_{j=1}^3\frac{\chi_1(q_j)}{q_j},\\
			|\det\nabla^2\Phi|\ge\tfrac38\quad\text{on }\supp\eta_\pm,\qquad
			\|A_\pm\|_{C^N}\le C_N.
		\end{gathered}
		\]
		Since $\nabla^2(\Phi+\xi r/\mu)=\nabla^2\Phi$, uniform stationary phase
		\cite[Corollary~1.1.8]{Sogge} for $|\xi|\le C_0\mu$ and two integrations
		by parts in $r$ for $|\xi|>C_0\mu$, with fixed sufficiently large $C_0$, give
		\[
		\sup_{\xi\in\R}
		\left|\int e^{i\mu\Phi(r,p)+i\xi r}A_\pm(r,p)\dd r\dd p\right|
		\le C\mu^{-2}.
		\]
		Choose a fixed $\rho\in C_c^\infty((1/8,2))$ equal to one on the
		projections of $\supp A_\pm$ onto the $r$ axis, and let
		$\beta(r)=\rho(r)r^{-1/2}\omega(\mu r/(a+b))$.
		Then $\|\widehat\beta\|_{L^1(\R)}\le C\|\beta\|_{H^1(\R)}\le C$,
		so Fourier inversion in $r$ gives
		\[
		\left|\int e^{i\mu\Phi(r,p)}A_\pm(r,p)\beta(r)\dd r\dd p\right|
		\le C\mu^{-2}\|\widehat\beta\|_1\le C\mu^{-2}.
		\]
		
		The remaining amplitude is
		\[
		\widetilde A=(1-\chi_0)(1-\eta_+-\eta_-)r^{-1/2}
		\omega\left(\frac{\mu r}{a+b}\right)
		\prod_{j=1}^3\frac{\chi_1(q_j)}{q_j}.
		\]
		The formulas for $\nabla_{r,p}\Phi$ extend continuously to
		$\alpha=0,1$, where their only zeros remain $(r_\alpha,\pm p_\alpha)$.
		Compactness and the preceding derivative bounds give, uniformly in $\alpha$,
		\[
		\inf_{\supp\widetilde A}|\nabla_{r,p}\Phi|\ge c>0,\qquad
		\sum_{|\gamma|\le2}\|\partial_{(r,p)}^\gamma\widetilde A\|_{L^1}\le C.
		\]
		Twice integrating by parts with
		$(i\mu|\nabla_{r,p}\Phi|^2)^{-1}\nabla_{r,p}\Phi\cdot\nabla_{r,p}$ gives
		\[
		\left|\int e^{i\mu\Phi}\widetilde A\dd r\dd p\right|
		\le C\mu^{-2}\sum_{|\gamma|\le2}
		\|\partial_{(r,p)}^\gamma\widetilde A\|_{L^1}
		\le C\mu^{-2}.
		\]
		This proves \eqref{eq:bessel-scaled-weight}. Multiplication by
		$(\mu/(a+b))^{1/2}$ proves \eqref{eq:global-triple-bound}.
	\end{proof}
	
	\subsection{Completion of the proof}\label{sec:global-localization}
	\begin{proof}[Proof of Theorem~\ref{thm:global}]
		By homogeneity, assume
		$\|e_\lambda\|_{L^2(M)}=\|w\|_{C_b^2(M)}=1$.
		Fix $0<4r<\delta$ and an even nonnegative
		$\eta_c\in C_c^\infty((-r,r))$ with
		$2\pi\int_0^r\eta_c(t)\sinh t\dd t=1$.
		Let $\pi_M:\HH\to M$ be the universal covering. For any lift
		$\widetilde p$ of $p$, let
		\[
		m_\lambda(p)=\|e_\lambda\circ\pi_M\|_{L^2(B_g(\widetilde p,2r))}.
		\]
		For $F\ge0$, apply Tonelli's theorem and unfold first with $\eta_c$,
		then with $\mathbf1_{[0,2r)}$, to obtain
		\begin{equation}\label{eq:global-unfolding}
			\begin{gathered}
				\int_M\int_{\HH}\eta_c(d_g(\widetilde p,z))F(\pi_Mz)
				\dd V_g(z)\dd V_g(p)=\int_MF\dd V_g,\\
				\int_Mm_\lambda(p)^2\dd V_g(p)=2\pi(\cosh(2r)-1).
			\end{gathered}
		\end{equation}
		If $\operatorname{inj}_M(p)\ge2r$, $\pi_M$ is injective on
		$B_g(\widetilde p,2r)$. Otherwise the image of this ball lies in a single
		cusp, because
		$\operatorname{inj}_M(q)<4r<\delta$ there. In the latter case,
		\[
		\#\bigl(B_g(\widetilde p,2r)\cap\pi_M^{-1}(q)\bigr)
		\le1+2y(p)\sinh(2r),\qquad m_\lambda(p)\le C_r y(p)^{1/2}.
		\]
		For $H\ge4y_0$, this proves
		\[
		\sup_{d_g(p,\supp w_H)<r}m_\lambda(p)\le C_{s_0}H^{1/2},\qquad
		\sup_\ell\sup_{\substack{p\in Y_\ell\\y(p)\ge e^{-r}H}}
		y(p)^{-1/2}m_\lambda(p)\le C_r.
		\]
		The unfolding identities \eqref{eq:global-unfolding} allow us to apply
		\eqref{eq:global-absolute-local} and then \eqref{eq:weight-c2}. The weight
		$\eta_c(d_g(\widetilde p,\cdot))[(w w_H)\circ\pi_M]$ has uniformly bounded
		covariant derivatives through order two, and
		$$\int_{\{d_g(p,\supp w_H)<r\}}m_\lambda(p)^3\dd V_g(p)
		\le C_{s_0}H^{1/2}.$$ Hence
		\begin{equation}\label{eq:global-shallow}
			\begin{aligned}
				\int_Mw_H|e_\lambda|^3\dd V_g
				&\le C_{s_0}\mu^{1/4}H^{1/2},\\
				\left|\int_Mw w_He_\lambda^3\dd V_g\right|
				&\le C_{s_0}\mu^{-1/6}H^{1/2}.
			\end{aligned}
		\end{equation}
		If $M$ has no cusps, then $w_H=1$. Taking $H=4y_0$ in
		\eqref{eq:global-shallow} gives $e_\lambda\in L^3(M)$ and
		\[
		\left|\int_Mw e_\lambda^3\dd V_g\right|
		\le C_{s_0}\mu^{-1/6}.
		\]
		This proves the conditional estimate in Remark~\ref{rem:cusp-free}
		and hence \eqref{eq:global} in the cusp-free case. We may therefore
		assume that $M$ has at least one cusp.
		
		For $H=4\max\{y_0,\lambda\}$, Lemma~\ref{lem:cusp-fourier} with $N=1$ gives
		\[
		\sum_\ell\int_{Y_\ell\cap\{y\ge H\}}|e_\lambda|^3\dd V_g
		\le C\sum_\ell\|e_\lambda\|_{L^2(Y_\ell)}^3\le C.
		\]
		Thus $e_\lambda\in L^3(M)$, all weighted integrals are absolutely
		convergent, and bounded frequencies satisfy the theorem.
		
		Now take $\mu$ large enough, depending only on $s_0$, and set
		\begin{equation}\label{eq:global-parameters}
			N=\lfloor\mu^{3/4}\rfloor,\qquad H=8\mu^{1/4}.
		\end{equation}
		In each cusp, let $\overline w(y)=\int_0^1w(x,y)\dd x$.
		Since $w-\overline w$ and $\partial_xw$ have mean zero in $x$,
		the fundamental theorem of calculus and the identities
		\[
		y^2\partial_x^2w
		=\nabla^2w(y\partial_x,y\partial_x)+y\partial_yw,\qquad
		(y\partial_y)^2w=\nabla^2w(y\partial_y,y\partial_y)
		\]
		give, uniformly over the cusps,
		\begin{equation}\label{eq:horocycle-weight}
			\begin{gathered}
				\|w(\cdot,y)-\overline w(y)\|_\infty
				\le\|\partial_xw(\cdot,y)\|_\infty
				\le\|\partial_x^2w(\cdot,y)\|_\infty\le Cy^{-2},\\
				\sum_{j=0}^2\|(y\partial_y)^j\overline w\|_\infty
				\le\sum_{j=0}^2\|(y\partial_y)^jw\|_\infty\le C.
			\end{gathered}
		\end{equation}
		In cusp coordinates, $|\nabla\log y|_g=1$ gives
		\[
		\left|\log\frac{y(z)}{y(p)}\right|\le d_g(\widetilde p,z),\qquad
		e^{-2r}y(p)\le y(z)\le e^{2r}y(p)
		\quad\bigl(z\in B_g(\widetilde p,2r)\bigr).
		\]
		On the support of
		$\eta_c(d_g(\widetilde p,z))\chi(y(z)/H)$, we also have
		$y(p)>e^{-r}H$. Thus \eqref{eq:global-absolute-local} and
		\eqref{eq:global-unfolding} give
		\begin{equation}\label{eq:horocycle-error}
			\begin{aligned}
				\sum_\ell\left|\int_{Y_\ell}\chi(y/H)(w-\overline w)e_\lambda^3\dd V_g\right|
				&\le C_r\mu^{1/4}\sum_\ell
				\int_{Y_\ell\cap\{y\ge e^{-r}H\}}y(p)^{-2}m_\lambda(p)^3\dd V_g(p)\\
				&\le C_r\mu^{1/4}H^{-3/2}\int_Mm_\lambda(p)^2\dd V_g(p)
				\le C_r\mu^{-1/8}.
			\end{aligned}
		\end{equation}
		
		Fix one cusp $Y$ and let $u=e_\lambda|_Y$.
		Set $\omega(y)=\chi(y/H)\overline w(y)$ on $y>y_0$ and extend it by
		zero to $(0,y_0]$. It is smooth and
		$\sum_{j=0}^2\|(y\partial_y)^j\omega\|_\infty\le C$ by
		\eqref{eq:horocycle-weight}.
		Let $u_{\le N}$ be the sum of the modes $0<|n|\le N$ in
		\eqref{eq:global-fourier}. Since $(N+1)H\ge8\mu$, Parseval and
		Lemma~\ref{lem:cusp-fourier} give
		\begin{equation}\label{eq:global-replace-tail}
			\begin{aligned}
				\left|\int_Y\omega(y)(u^3-u_{\le N}^3)\dd V_g\right|
				\le C\|u-u_{\le N}\|_{L^\infty(\{y\ge H\})}\|u\|_{L^2(Y)}^2\le Ce^{-c\mu}\|u\|_{L^2(Y)}^3.
			\end{aligned}
		\end{equation}
		For $0<|n|\le N$ and $y_0\le y\le\mu/(8\pi N)$, we have
		$2\pi|n|y\le\mu/4$ and
		$\log(\mu/(8\pi Ny_0))=\tfrac14\log\mu+O_{y_0}(1)$.
		Thus Parseval and \eqref{eq:global-mass} imply
		\begin{equation}\label{eq:global-coefficients}
			\|u\|_{L^2(Y)}^2
			\ge c\mu^{-1}\log\mu\sum_{0<|n|\le N}|c_n|^2.
		\end{equation}
		In the following sums set $c_n=0$ for $n=0$ or $|n|>N$.
		Integration in $x$ retains $n_1+n_2+n_3=0$. Up to permutation and
		simultaneous change of sign, these triples are $(n_1,n_2,-n_1-n_2)$
		with $n_1,n_2\ge1$. Cauchy--Schwarz gives
		\[
		\begin{aligned}
			\sum_{n_1,n_2\ge1}\frac{|c_{n_1}c_{n_2}c_{-n_1-n_2}|}{\sqrt{n_1+n_2}}
			\le\left(\sum_{n_1,n_2\ge1}|c_{n_1}c_{n_2}|^2\right)^{1/2}
			\left(\sum_{k\ge2}\frac{k-1}{k}|c_{-k}|^2\right)^{1/2}\le\left(\sum_{n\ne0}|c_n|^2\right)^{3/2}.
		\end{aligned}
		\]
		The reversed signs obey the same bound. Lemma~\ref{lem:bessel-triple},
		with $a=2\pi n_1$, $b=2\pi n_2$, therefore yields
		\begin{equation}\label{eq:global-full-cubic}
			\begin{aligned}
				\left|\int_Y\omega(y)u_{\le N}^3\dd V_g\right|
				\le C\mu^{-3/2}\left(\sum_{0<|n|\le N}|c_n|^2\right)^{3/2}\le C(\log\mu)^{-3/2}\|u\|_{L^2(Y)}^3.
			\end{aligned}
		\end{equation}
		Finally, \eqref{eq:global-cutoff}, \eqref{eq:global-shallow},
		\eqref{eq:horocycle-error}, \eqref{eq:global-replace-tail}, and
		\eqref{eq:global-full-cubic} give
		\begin{equation}\label{eq:global-cusp-sum}
			\begin{aligned}
				\left|\int_Mw e_\lambda^3\dd V_g\right|
				\le C_{s_0}(\mu^{-1/24}+\mu^{-1/8})
				+C(\log\mu)^{-3/2}\sum_\ell\|e_\lambda\|_{L^2(Y_\ell)}^3\le C_{s_0}(\log\mu)^{-3/2},
			\end{aligned}
		\end{equation}
		since $\sum_\ell\|e_\lambda\|_{L^2(Y_\ell)}^3
		\le(\sum_\ell\|e_\lambda\|_{L^2(Y_\ell)}^2)^{3/2}\le1$.
		This proves \eqref{eq:global} because $\mu\approx\lambda$.
		
		For the mixed bound, the polarization identity is
		\begin{equation}\label{eq:polarization}
			24u_1u_2u_3
			=\sum_{\substack{\sigma\in\{-1,1\}^3\\\sigma_1\sigma_2\sigma_3=1}}
			(\sigma_1u_1+\sigma_2u_2+\sigma_3u_3)^3.
		\end{equation}
		Normalize each nonzero $u_j$ in $L^2(M)$. Each sum then satisfies
		the same eigenvalue equation and has $L^2$ norm at most $3$. The
		cubic bound applied to these sums gives the normalized mixed
		estimate. 
	\end{proof}
	
	\section{Sharpness}\label{sec:sharpness}
	
We give examples establishing
	the sharpness of Theorem~\ref{thm:upper}, with the metric, weight,
	and normalization region fixed as the frequency tends to infinity.

	\subsection{Zero curvature}
	\begin{proposition}[Sharpness of Theorem~\ref{thm:upper}\textup{(i)}]\label{prop:sharp-flat}
		\textup{(i)} On $\R^2$ with $g=\dd x^2+\dd y^2$, let
		$U=(-1/2,1/2)^2$. For every nonzero $0\le w\in C_c^\infty(U)$,
		there are real $e_\lambda\in C^\infty(\R^2)$, $\lambda\ge1$, such that
		\[
		-\Delta_g e_\lambda=\lambda^2e_\lambda\quad\text{on }\R^2,\qquad
		\|e_\lambda\|_{L^2(U)}=1,\qquad
		\lim_{\lambda\to\infty}\int_Uw e_\lambda^3\dd x\dd y
		=\sqrt{\frac23}\int_Uw\dd x\dd y>0.
		\]
		
		\textup{(ii)} On the flat torus
		\[
		M=\bigl(\R/(3^{-1/4}\mathbb Z)\bigr)
		\times\bigl(\R/(3^{1/4}\mathbb Z)\bigr),\qquad
		g=\dd x^2+\dd y^2,
		\]
		which has area one, there are real eigenfunctions $e_n$ with
		$\lambda_n\to\infty$ such that
		\begin{equation}\label{eq:flat-sharpness}
			-\Delta_ge_n=\lambda_n^2e_n,\qquad
			\|e_n\|_{L^2(M)}=1,\qquad
			\int_Me_n^3\dd V_g=\sqrt{\frac23}.
		\end{equation}
	\end{proposition}
	\begin{proof}
		For both parts use
		\[
		u_\lambda(x,y)=\cos(\lambda y)
		+\cos\left(\frac{\lambda(\sqrt3x-y)}2\right)
		+\cos\left(\frac{\lambda(-\sqrt3x-y)}2\right).
		\]
		The three phase gradients have length $\lambda$ and sum to zero. Thus
		$-\Delta_gu_\lambda=\lambda^2u_\lambda$.
		
		\textup{(i)} Apart from the constant terms, every Fourier mode of $u_\lambda^2$
		and $u_\lambda^3$ has frequency $\lambda\xi$ with fixed $\xi\ne0$.
		Integration by parts against $w$ gives
		\[
		\|u_\lambda\|_{L^2(U)}^2=\frac32+O(\lambda^{-1}),\qquad
		\int_Uw u_\lambda^3\dd x\dd y
		=\frac32\int_Uw\dd x\dd y+O_{w,N}(\lambda^{-N})
		\]
		for every $N$. Since $u_\lambda(0,0)=3$, set
		$e_\lambda=u_\lambda/\|u_\lambda\|_{L^2(U)}$ to obtain the stated limit.
		
		\textup{(ii)} For $\lambda_n=4\pi n3^{-1/4}$, the function $u_{\lambda_n}$
		is periodic on $M$. All its nonconstant Fourier terms integrate to zero, so
		\[
		\int_Mu_{\lambda_n}^2\dd x\dd y
		=\int_Mu_{\lambda_n}^3\dd x\dd y=\frac32.
		\]
		Thus $e_n=\sqrt{2/3}\,u_{\lambda_n}$ proves \eqref{eq:flat-sharpness}.
	\end{proof}

	\subsection{Arbitrarily slow convergence}
	\begin{proposition}[Sharpness of Theorem~\ref{thm:upper}\textup{(ii)}]
		\label{prop:sharp-slow}
		There are fixed $U\Subset\R^2$ and nonzero $0\le w\in C_c^\infty(U)$
		with the following property. For every
		$\varepsilon_*:[1,\infty)\to(0,\infty)$ with $\varepsilon_*(\lambda)\to0$,
		there are a fixed complete smooth metric $g$ on $\R^2$ and real solutions
		\begin{equation}\label{eq:sharp-normalization}
			-\Delta_ge_\lambda=\lambda^2e_\lambda\quad\text{on }U,
			\qquad \|e_\lambda\|_{L^2(U)}=1,
		\end{equation}
		such that, for some $\lambda_n\to\infty$,
		\begin{equation}\label{eq:sharp-ratio}
			\int_Uw e_\lambda^3\dd V_g\longrightarrow0,\qquad
			\frac{\displaystyle\int_Uw e_{\lambda_n}^3\dd V_g}
			{\varepsilon_*(\lambda_n)}\longrightarrow+\infty.
		\end{equation}
		Moreover,
		\[
		w(x,0)=0,\qquad K_g(x,y)<0\quad(y\ne0),\qquad
		\nabla^jK_g(x,0)=0\quad(j\ge0).
		\]
	\end{proposition}
	\begin{proof}
		For the local constructions, fix
		\begin{equation}\label{eq:fixed-U-w}
			U=(-1/2,1/2)^2,\qquad
			0\le\psi_1\in C_c^\infty((-1/4,1/4)),\qquad
			\psi_1(0)>0,\quad\int_\R\psi_1=1.
		\end{equation}
		For this construction, choose $0\le\psi_2\in C_c^\infty((-1/4,1/4))$ with
		$\psi_2(y)=y^2$ for $|y|\le1/8$, and set $w(x,y)=\psi_1(x)\psi_2(y)$.
		
		For $a\in C^\infty(\R)$ with $a\ge15/16$, $a(0)=1$, and $a'(0)=0$,
		the metric and curvature formulas \cite[Example~3.2.12]{Petersen} are
		\begin{equation}\label{eq:warped-metric}
			\begin{gathered}
				g_a=\dd y^2+a(y)^2\dd x^2,\qquad dV_{g_a}=a\dd x\dd y,\\
				\Delta_{g_a}=\partial_y^2+(a'/a)\partial_y+a^{-2}\partial_x^2,
				\qquad K_{g_a}=-a''/a.
			\end{gathered}
		\end{equation}
		The metric is complete because $g_a\ge(15/16)^2(\dd x^2+\dd y^2)$.
		Fix $\xi_*=\sqrt3/2<\xi_0<15/16$. For $|\xi|\le\xi_0$, let
		\[
		q_\xi=1-\xi^2a^{-2},\qquad
		s_\xi(y)=\int_0^y\sqrt{q_\xi(t)}\dd t,\qquad
		b_\xi=a^{-1/2}\left(\frac{1-\xi^2}{q_\xi}\right)^{1/4},
		\]
		and let $v_{\lambda,\xi}$ solve
		\begin{equation}\label{eq:exact-ode}
			v_{\lambda,\xi}''+(a'/a)v_{\lambda,\xi}'+\lambda^2q_\xi v_{\lambda,\xi}=0,
			\qquad v_{\lambda,\xi}(0)=1,\quad
			v_{\lambda,\xi}'(0)=i\lambda\sqrt{1-\xi^2}.
		\end{equation}
		Then $e^{i\lambda\xi x}v_{\lambda,\xi}$ and
		$e^{i\lambda\xi x}\overline{v_{\lambda,\xi}}$ are exact solutions.
		Eliminating the first-derivative term gives
		\begin{equation}\label{eq:common-ode}
			(a^{1/2}v_{\lambda,\xi})''+
			\left(\lambda^2q_\xi+\frac{(a')^2}{4a^2}-\frac{a''}{2a}\right)
			a^{1/2}v_{\lambda,\xi}=0.
		\end{equation}
		Since $q_\xi\ge1-(16\xi_0/15)^2>0$, the oscillatory
		Liouville--Green approximation \cite[Chapter~6]{Olver}, matched to
		\eqref{eq:exact-ode} using $b_\xi(0)=1$ and $b_\xi'(0)=0$, gives uniformly
		for $|\xi|\le\xi_0$ and $|y|\le1/2$,
		\begin{equation}\label{eq:ode-asymptotics}
			v_{\lambda,\xi}(y)=b_\xi(y)e^{i\lambda s_\xi(y)}+O_a(\lambda^{-1}),
			\qquad |v_{\lambda,\xi}(y)|\le C_a.
		\end{equation}
		These constants are uniform when $\|a-1\|_{C^2([-1,1])}\le1$.
		
		Define the real exact solution
		\[
		u_\lambda(x,y)=\operatorname{Re}v_{\lambda,0}(y)
		+2\cos(\lambda\xi_*x)\operatorname{Re}v_{\lambda,\xi_*}(y).
		\]
		Then $u_\lambda(0,0)=3$ and $\|u_\lambda\|_{L^2(U,g_a)}\le C_a$.
		Let
		\begin{equation}\label{eq:reduced-moment}
			\phi_a=2s_{\xi_*}-y,\qquad
			\beta_a=\frac{\psi_2}{2\sqrt{a q_{\xi_*}}}.
		\end{equation}
		Since $\int_\R\psi_1(x)e^{ik\lambda\xi_*x}\dd x=O_N(\lambda^{-N})$
		for $k=\pm1,\pm2,\pm3$, expanding the cube and using
		\eqref{eq:ode-asymptotics} gives
		\[
		\begin{aligned}
			\int_Uw u_\lambda^3\dd V_{g_a}
			&=\int_\R\psi_2 a\bigl[b_0^3\cos^3(\lambda y)
			+6b_0b_{\xi_*}^2\cos(\lambda y)\cos^2(\lambda s_{\xi_*})\bigr]\dd y
			+O_a(\lambda^{-1}).
		\end{aligned}
		\]
		Use
		\[
		6\cos A\cos^2 B=3\cos A+\tfrac32\cos(A+2B)+\tfrac32\cos(A-2B),
		\qquad a b_0b_{\xi_*}^2=\frac1{2\sqrt{a q_{\xi_*}}}.
		\]
		The phases $y$, $3y$, and $y+2s_{\xi_*}(y)$ have derivatives bounded
		away from zero. One integration by parts in $y$ yields
		\begin{equation}\label{eq:cubic-asymptotics}
			\int_Uw u_\lambda^3\dd V_{g_a}
			=\frac32\int_\R\beta_a(y)\cos(\lambda\phi_a(y))\dd y
			+O_a(\lambda^{-1}).
		\end{equation}
		For sufficiently large integers $n$, set
		\begin{equation}\label{eq:slow-radii}
			r_n=\left(\sup_{\lambda\ge n}\varepsilon_*(\lambda)+n^{-1}\right)^{1/6}.
		\end{equation}
		Choose $n_0$ so that $r_n<1/8$ for $n\ge n_0$. Then $r_n\downarrow0$,
		$\sup_{\lambda\ge n}\varepsilon_*(\lambda)\le r_n^6$, and $n^{-1}\le r_n^6$.
		Choose a smooth $0\le\chi\le1$ that vanishes on $(-\infty,1]$ and is
		positive on $(1,\infty)$, with $\chi=1$ on $[2,\infty)$. Define
		\begin{equation}\label{eq:slow-metric}
			a(y)=1+\sum_{n\ge n_0}2^{-n}r_n^n
			\int_0^{|y|}(|y|-s)\chi(s/r_n)\dd s.
		\end{equation}
		For every fixed $j\ge0$, the derivatives of order $j$ of the series for
		$a''$ converge uniformly, since
		$\sum_{n\ge n_0}2^{-n}r_n^{n-j}<\infty$. Each summand vanishes near zero.
		Consequently
		\begin{equation}\label{eq:limiting-curvature}
			\begin{gathered}
				a\ge1,\qquad \|a-1\|_{C^2([-1,1])}\le1,\qquad a^{(j)}(0)=0\quad(j\ge1),\\
				a''(y)=\sum_{n\ge n_0}2^{-n}r_n^n\chi(|y|/r_n)>0\quad(y\ne0).
			\end{gathered}
		\end{equation}
		For $|y|\le r_n$, only indices greater than $n$ contribute, so
		\[
		0\le a(y)-1\le 2^{-n-1}y^2,\qquad
		0\le\phi_a'(y)=2\sqrt{1-3/(4a(y)^2)}-1\le3(a(y)-1).
		\]
		Thus, for large $n$,
		\begin{equation}\label{eq:slow-phase}
			|\phi_a(y)|\le2^{-n-1}|y|^3\le n^{-2},\qquad
			\beta_a(y)\ge c y^2\qquad(|y|\le r_n).
		\end{equation}
		The Fourier transform of a Gaussian gives
		\[
		\frac{2}{\sqrt\pi\,n^2}\int_0^\infty
		e^{-\lambda^2/n^4}\cos(\lambda\phi_a(y))\dd\lambda
		=e^{-n^4\phi_a(y)^2/4}.
		\]
		The contribution of $0\le\lambda<n$ is $O(n^{-1})$ after integration
		against $\beta_a$. On $\lambda\ge n$, the error in
		\eqref{eq:cubic-asymptotics} is at most $C_a n^{-1}$. Hence
		\eqref{eq:slow-phase} and Fubini give
		\begin{equation}\label{eq:positive-average}
			\begin{aligned}
				&\frac{2}{\sqrt\pi\,n^2}\int_n^\infty e^{-\lambda^2/n^4}
				\left(\int_Uw u_\lambda^3\dd V_{g_a}\right)\dd\lambda\\
				&\quad=\frac32\int_\R\beta_a(y)e^{-n^4\phi_a(y)^2/4}\dd y+O_a(n^{-1})\\
				&\quad\ge cr_n^3-C_a n^{-1}\ge c'r_n^3.
			\end{aligned}
		\end{equation}
		The last inequality uses $n^{-1}\le r_n^6=o(r_n^3)$. The nonnegative
		averaging weight has integral at most one, so some $\lambda_n\ge n$ satisfies
		\begin{equation}\label{eq:chosen-frequency}
			\int_Uw u_{\lambda_n}^3\dd V_{g_a}\ge c r_n^3.
		\end{equation}
		For every $\lambda\ge1$, set $e_\lambda=u_\lambda/\|u_\lambda\|_{L^2(U,g_a)}$.
		Since $u_\lambda(0,0)=3$ and $\|u_\lambda\|_{L^2(U,g_a)}\le C_a$,
		\[
		\frac{\displaystyle\int_Uw e_{\lambda_n}^3\dd V_{g_a}}
		{\varepsilon_*(\lambda_n)}
		\ge c\frac{r_n^3}{\varepsilon_*(\lambda_n)}
		\ge c r_n^{-3}\longrightarrow\infty.
		\]
		Equation \eqref{eq:limiting-curvature} gives $\{K_{g_a}=0\}=\{y=0\}$ and
		$\nabla^jK_{g_a}(x,0)=0$ for all $j$. Theorem~\ref{thm:upper}\textup{(ii)}
		gives convergence to zero for the whole family.
	\end{proof}
	
	\subsection{Finite-order vanishing of curvature}\label{sec:sharp-finite}
	\begin{proposition}[Sharpness of Theorem~\ref{thm:upper}\textup{(iii)}]
		\label{prop:sharp-finite}
		\textup{(i)} For the fixed rectangle $U$ in \eqref{eq:fixed-U-w}, there
		is a fixed $0\le w\in C_c^\infty(U)$ with $w(0)>0$ such that, for every
		integer $m\ge0$, a fixed complete smooth metric $g_m$ satisfies
		\begin{equation}\label{eq:sharp-finite-condition}
			\inf_{x\in\supp w}\sum_{j=0}^m|\nabla^jK_{g_m}(x)|_{g_m}>0
		\end{equation}
		and admits real local solutions with
		\begin{equation}\label{eq:power-sharpness}
			\begin{gathered}
				-\Delta_{g_m}e_\lambda=\lambda^2e_\lambda\quad\text{on }U,\qquad
				\|e_\lambda\|_{L^2(U,g_m)}=1,\\
				c_m\lambda^{-1/(6+2m)}\le\int_Uw e_\lambda^3\dd V_{g_m}
				\le C_m\lambda^{-1/(6+2m)}
			\end{gathered}
		\end{equation}
		for all sufficiently large $\lambda$.
		
		\textup{(ii)} On the standard round sphere $S^2$, for every
		sufficiently large even $n$, there is a real spherical harmonic $e_n$ with
		\begin{equation}\label{eq:sphere-sharpness}
			\begin{gathered}
				-\Delta_{S^2}e_n=\lambda_n^2e_n,\qquad
				\lambda_n=\sqrt{n(n+1)},\qquad\|e_n\|_{L^2(S^2)}=1,\\
				c\lambda_n^{-1/6}\le\int_{S^2}e_n^3\dd V_g
				\le C\lambda_n^{-1/6}.
			\end{gathered}
		\end{equation}
	\end{proposition}
	\begin{proof}
		For complex solutions $u_j$ with $\|u_j\|_2\le C$, expansion into real
		and imaginary parts gives a real product with integral of absolute value
		at least $\frac18|\int w\prod_{j=1}^3u_j|$.
		Equation \eqref{eq:polarization} gives a real linear combination with
		cubic integral of absolute value at least six times this value and
		$L^2$ norm at most $3C$. Choosing its sign and normalizing gives
		\begin{equation}\label{eq:real-conversion}
			\|e_\lambda\|_2=1,\qquad
			\int w e_\lambda^3\dd V_g
			\ge\frac1{36C^3}\left|\int w\prod_{j=1}^3u_j\dd V_g\right|
		\end{equation}
		whenever the integral on the right is nonzero.
		
		\textup{(i)} Choose an even $\chi\in C_c^\infty((-1,1))$ with
		$0\le\chi\le1$ and $\chi=1$ on $[-1/2,1/2]$, and let
		\[
		a_m(y)=1+\chi(4y)y^{m+2},\qquad g_m=\dd y^2+a_m(y)^2\dd x^2.
		\]
		For this construction, retain $\psi_1$ from \eqref{eq:fixed-U-w} and
		choose $0\le\psi_2\in C_c^\infty((-1/16,1/16))$ equal to one
		near zero. Let $w(x,y)=\psi_1(x)\psi_2(y)$.
		These $U,w$ are independent of $m$. The metric is complete since $a_m\ge15/16$.
		On $|y|<1/8$,
		\begin{equation}\label{eq:finite-sharp-curvature}
			K_{g_m}(y)=-\frac{(m+1)(m+2)y^m}{1+y^{m+2}},\qquad
			\partial_y^mK_{g_m}(0)=-(m+1)(m+2)m!.
		\end{equation}
		The curves $y\mapsto(x_0,y)$ are unit-speed geodesics, so
		$(\nabla^mK_{g_m})(\partial_y,\ldots,\partial_y)=\partial_y^mK_{g_m}$
		for $m\ge1$. Since $K_{g_m}(y)\ne0$ for $y\ne0$, compactness gives
		\eqref{eq:sharp-finite-condition}.
		
		Choose an even $G\in\mathcal S(\R)$ with $G\ge0$, $G(0)>0$, and
		$\widehat G\in C_c^\infty(\R)$, with Fourier normalization
		$\widehat G(t)=(2\pi)^{-1/2}\int e^{-its}G(s)\dd s$.
		Such $G$ is the square of the inverse Fourier transform of a nonzero
		even nonnegative smooth function of compact support.
		Use \eqref{eq:exact-ode} with $a=a_m$. For fixed small $c_*>0$, define
		\begin{equation}\label{eq:common-average}
			\begin{gathered}
				\delta_m=c_*\lambda^{-1/(m+3)},\\
				u_{1,\lambda}(x,y)=\frac{\delta_m^{-1/2}}{\sqrt{2\pi}}
				\int\widehat G(t)e^{itx/\delta_m}v_{\lambda,t/(\lambda\delta_m)}(y)\dd t,\\
				u_{2,\lambda}(x,y)=\frac{\delta_m^{-1/2}}{\sqrt{2\pi}}
				\int\widehat G(t)e^{i(\lambda\xi_*+t/\delta_m)x}
				\overline{v_{\lambda,\xi_*+t/(\lambda\delta_m)}(y)}\dd t,
				\qquad u_{3,\lambda}(x,y)=u_{2,\lambda}(-x,y).
			\end{gathered}
		\end{equation}
		All parameters lie in $(-\xi_0,\xi_0)$ for large $\lambda$.
		These are exact solutions because $a_m$ is independent of $x$.
		For fixed $|y|\le1/2$, Plancherel and \eqref{eq:ode-asymptotics} give
		\[
		\int_\R|u_{1,\lambda}(x,y)|^2\dd x
		=\int_\R|\widehat G(t)|^2
		|v_{\lambda,t/(\lambda\delta_m)}(y)|^2\dd t\le C_m.
		\]
		The same bound holds for $u_{2,\lambda}$ and $u_{3,\lambda}$, so
		$\|u_{j,\lambda}\|_{L^2(U,g_m)}\le C_m$ for $j=1,2,3$.
		
		Let
		\[
		\ell(x,y)=x-\left.\partial_\xi s_\xi(y)\right|_{\xi=\xi_*}
		=x+\xi_*\int_0^y\frac{\dd t}{a_m(t)^2\sqrt{q_{\xi_*}(t)}}.
		\]
		Then $\partial_x\ell=1$, $\partial_y\ell\ge c_m>0$, and
		$|(x,y)|\le C_m(|x|+|\ell(x,y)|)$ on $U$.
		Substitute \eqref{eq:ode-asymptotics} into \eqref{eq:common-average}.
		The first $t$ derivatives of the two phases are respectively
		\[
		\frac{x}{\delta_m}+O_m((\lambda\delta_m^2)^{-1}),\qquad
		\frac{\ell(x,y)}{\delta_m}+O_m((\lambda\delta_m^2)^{-1}).
		\]
		Their higher derivatives and the amplitude derivatives are uniformly
		bounded, since $\lambda\delta_m^2\to\infty$. Three integrations by parts
		in $t$ bound the two leading integrals by
		$C_m\delta_m^{-1/2}(1+|x|/\delta_m)^{-3}$ and
		$C_m\delta_m^{-1/2}(1+|\ell|/\delta_m)^{-3}$.
		The remainder is $O_m(\lambda^{-1}\delta_m^{-1/2})$ and is absorbed into
		these bounds because $x,\ell$ are bounded on $U$ and
		$\lambda^{-1}\le c_*^{-3}\delta_m^3$. Consequently
		\begin{equation}\label{eq:common-average-decay}
			\begin{gathered}
				|u_{1,\lambda}(x,y)|\le C_{m,c_*}\delta_m^{-1/2}
				(1+|x|/\delta_m)^{-3},\\
				|u_{2,\lambda}(x,y)|\le C_{m,c_*}\delta_m^{-1/2}
				(1+|\ell(x,y)|/\delta_m)^{-3},\\
				\delta_m^{3/2}\prod_{j=1}^3|u_{j,\lambda}(\delta_m z)|
				\le C_{m,c_*}(1+|z|)^{-3}.
			\end{gathered}
		\end{equation}
		In the last line the product is extended by zero outside $\delta_m^{-1}U$.
		
		On compact sets of $(x,y)$, Taylor's formula in $\xi$ and Fourier
		inversion give
		\[
		\begin{aligned}
			\delta_m^{1/2}e^{-i\lambda\delta_m y}
			u_{1,\lambda}(\delta_m x,\delta_m y)&\longrightarrow G(x),\\
			\delta_m^{1/2}e^{-i\lambda[\xi_*\delta_m x-s_{\xi_*}(\delta_m y)]}
			u_{2,\lambda}(\delta_m x,\delta_m y)&\longrightarrow G(x+\sqrt3y).
		\end{aligned}
		\]
		Indeed $b_\xi(\delta_m y)\to1$, the quadratic Taylor error in the phase
		is $O_m((\lambda\delta_m)^{-1})$, and
		$\ell(\delta_m x,\delta_m y)/\delta_m\to x+\sqrt3y$.
		Reflection gives the third limit with $G(x-\sqrt3y)$.
		Moreover,
		\[
		y-2s_{\xi_*}(y)=-\frac3{m+3}y^{m+3}+O_m(y^{2m+5}),\qquad
		\lambda\delta_m^{m+3}=c_*^{m+3}.
		\]
		Since $(1+|z|)^{-3}\in L^1(\R^2)$, dominated convergence after
		$z\mapsto\delta_m z$ gives
		\begin{equation}\label{eq:finite-sharp-limit}
			\begin{aligned}
				&\lim_{\lambda\to\infty}\delta_m^{-1/2}
				\operatorname{Re}\int_Uw\prod_{j=1}^3u_{j,\lambda}\dd V_{g_m}\\
				&\quad=w(0)\int_{\R^2}G(x)G(x+\sqrt3y)G(x-\sqrt3y)
				\cos\left(\frac{3c_*^{m+3}}{m+3}y^{m+3}\right)\dd x\dd y>0.
			\end{aligned}
		\end{equation}
		The last integral is positive for sufficiently small $c_*>0$, by
		continuity at $c_*=0$.
		Equation~\eqref{eq:real-conversion} gives the lower bound in
		\eqref{eq:power-sharpness}. Theorem~\ref{thm:upper}\textup{(iii)} gives the upper bound.
		
		\medskip
		\textup{(ii)} Use the same cutoff $\chi$. For $v=(v_1,v_2)\in\mathbb R^2$, let
		$v^\perp=(-v_2,v_1)$. For even $n\ge2$, let
		\[
		\begin{gathered}
			\delta_0=c_*n^{-1/3},\qquad
			\theta_j=2\pi(j-1)/3,\qquad v_j=(\cos\theta_j,\sin\theta_j),\\
			\zeta_j(t)=(\sin t\,v_j^\perp+i v_j,\cos t)\in\mathbb C^3.
		\end{gathered}
		\]
		The dot product on $\mathbb C^3$ is the complex bilinear extension
		of the Euclidean inner product. Since
		$\zeta_j(t)\cdot\zeta_j(t)=0$,
		\[
		\Delta_{\R^3}(\zeta_j(t)\cdot x)^n
		=n(n-1)(\zeta_j(t)\cdot\zeta_j(t))(\zeta_j(t)\cdot x)^{n-2}=0.
		\]
		Thus the functions
		\begin{equation}\label{eq:sphere-average}
			u_{j,n}(x)=\left(\frac{n}{\delta_0}\right)^{1/2}
			\int_\R\chi(t/\delta_0)(\zeta_j(t)\cdot x)^n\dd t,
			\qquad x\in S^2,
		\end{equation}
		satisfy $-\Delta_{S^2}u_{j,n}=n(n+1)u_{j,n}$.
		
		For $\xi_1,\ldots,\xi_{2k}\in\mathbb C^3$, Isserlis's formula
		\cite[Theorem~1.28]{Janson}, extended by complex multilinearity, and
		polar integration give
		\[
		\begin{gathered}
			(2\pi)^{-3/2}\int_0^\infty r^{2k+2}e^{-r^2/2}\dd r
			=\frac{(2k+1)!}{4\pi\,2^k k!},\\
			\int_{S^2}\prod_{j=1}^{2k}(\xi_j\cdot x)\dd V_g(x)
			=\frac{4\pi\,2^k k!}{(2k+1)!}
			\sum_{\pi}\prod_{\{i,j\}\in\pi}(\xi_i\cdot\xi_j).
		\end{gathered}
		\]
		Here $k\ge0$ is an integer and the sum is over all pairings of
		$\{1,\ldots,2k\}$. Since $\zeta_j\cdot\zeta_j=0$, only the $n!$
		pairings between the two groups contribute to the $L^2$ integral.
		Using $\zeta_j(t)\cdot\overline{\zeta_j(s)}=2\cos^2((t-s)/2)$ gives
		\[
		\begin{aligned}
			&\int_{S^2}(\zeta_j(t)\cdot x)^n
			(\overline{\zeta_j(s)}\cdot x)^n\dd V_g(x)\\
			&\quad=\frac{4^{n+1}\pi(n!)^2}{(2n+1)!}
			\cos^{2n}\left(\frac{t-s}{2}\right)
			\le Cn^{-1/2}e^{-cn(t-s)^2},\qquad |t|,|s|\le1.
		\end{aligned}
		\]
		Stirling's formula gives the last bound. Young's inequality yields
		\begin{equation}\label{eq:sphere-norm}
			\begin{aligned}
				\|u_{j,n}\|_{L^2(S^2)}^2
				&\le\frac{C\sqrt n}{\delta_0}
				\iint\chi(t/\delta_0)\chi(s/\delta_0)e^{-cn(t-s)^2}\dd t\dd s\\
				&\le\frac{C\sqrt n}{\delta_0}
				\|\chi(\cdot/\delta_0)\|_2^2\|e^{-cn(\cdot)^2}\|_1\le C.
			\end{aligned}
		\end{equation}
		
		For $t=(t_1,t_2,t_3)$ near zero, define
		\[
		\mathcal A(t)=\frac8{27}
		\prod_{1\le i<j\le3}\bigl(\zeta_i(t_i)\cdot\zeta_j(t_j)\bigr),
		\qquad \mathcal A(0)=1.
		\]
		In the triple product, let $m_{ij}$ count the pairs joining groups
		$i,j$. Since $\zeta_j\cdot\zeta_j=0$, every nonzero term satisfies
		\[
		m_{12}+m_{13}=m_{12}+m_{23}=m_{13}+m_{23}=n,
		\qquad m_{12}=m_{13}=m_{23}=n/2.
		\]
		The number of such pairings is
		\[
		\binom{n}{n/2}^{\!3}\bigl((n/2)!\bigr)^3
		=\frac{(n!)^3}{((n/2)!)^3}.
		\]
		Isserlis's formula and polar integration with $2k=3n$ give
		\begin{equation}\label{eq:sphere-triple-moment}
			\int_{S^2}\prod_{j=1}^3(\zeta_j(t_j)\cdot x)^n\dd V_g(x)
			=\frac{4\pi\,3^{3n/2}(n!)^3(3n/2)!}
			{(3n+1)!((n/2)!)^3}\mathcal A(t)^{n/2}.
		\end{equation}
		The coefficient is positive and comparable to $n^{-1}$ by Stirling's formula.
		For $i\ne j$,
		\[
		\zeta_i(t_i)\cdot\zeta_j(t_j)
		=\cos t_i\cos t_j-\tfrac12\sin t_i\sin t_j+\tfrac12
		+i\sin(\theta_j-\theta_i)(\sin t_i-\sin t_j).
		\]
		Hence $\partial_{t_i}\log\mathcal A(0)=0$ and
		$\partial_{t_i}\partial_{t_j}\log\mathcal A(0)=-2/3$ for $1\le i,j\le3$.
		With the branch $\log\mathcal A(0)=0$, Taylor's formula gives
		\begin{equation}\label{eq:sphere-phase-expansion}
			\log\mathcal A(t)=-\frac{(t_1+t_2+t_3)^2}{3}+O(|t|^3).
		\end{equation}
		On the support of $\prod_{j=1}^3\chi(t_j/\delta_0)$, choose $c_*>0$ small enough that
		\[
		\frac n2|\operatorname{Im}\log\mathcal A(t)|
		\le Cn\delta_0^3=Cc_*^3\le\pi/3,
		\qquad \operatorname{Re}(\mathcal A(t)^{n/2})
		\ge\tfrac12|\mathcal A(t)|^{n/2}\ge0.
		\]
		\begin{samepage}
			For sufficiently large $n$, the set
			\[
			|t_1|,|t_2|\le\delta_0/8,\qquad
			|t_1+t_2+t_3|\le n^{-1/2}
			\]
			has volume $\delta_0^2/(8\sqrt n)$ and satisfies $|t_j|\le\delta_0/2$.
			The weight equals one there, and \eqref{eq:sphere-phase-expansion} gives
			\[
			|\mathcal A(t)|^{n/2}
			\ge\exp\left[-C\bigl(n(t_1+t_2+t_3)^2+n\delta_0^3\bigr)\right]\ge c.
			\]
		\end{samepage}
		Fubini, \eqref{eq:sphere-triple-moment}, and nonnegativity on the entire
		integration set therefore imply
		\begin{equation}\label{eq:sphere-global-lower}
			\operatorname{Re}\int_{S^2}\prod_{j=1}^3u_{j,n}\dd V_g
			\ge c\left(\frac n{\delta_0}\right)^{3/2}
			n^{-1}\delta_0^2n^{-1/2}=c\delta_0^{1/2}.
		\end{equation}
		Together with \eqref{eq:sphere-norm}, \eqref{eq:real-conversion} gives real
		normalized eigenfunctions with cubic integral at least
		$c\delta_0^{1/2}\approx\lambda_n^{-1/6}$.
		Theorem~\ref{thm:upper}\textup{(iii)} gives the upper bound.
	\end{proof}


\begin{thebibliography}{99}
		
		\bibitem{AurichSteiner}
		R. Aurich and F. Steiner,
		\emph{Statistical properties of highly excited quantum eigenstates of a strongly chaotic system},
		Phys. D \textbf{64} (1993), nos.~1--3, 185--214.
		
		\bibitem{BenedettiPetronio}
		R. Benedetti and C. Petronio,
		\emph{Lectures on Hyperbolic Geometry},
		Universitext, Springer-Verlag, Berlin, 1992.
		
		\bibitem{BernsteinReznikov1999}
		J. Bernstein and A. Reznikov,
		\emph{Analytic continuation of representations and estimates of automorphic forms},
		Ann. of Math. (2) \textbf{150} (1999), no.~1, 329--352.
		
		\bibitem{BernsteinReznikov2010}
		J. Bernstein and A. Reznikov,
		\emph{Subconvexity bounds for triple $L$-functions and representation theory},
		Ann. of Math. (2) \textbf{172} (2010), no.~3, 1679--1718.
		
		\bibitem{Berry}
		M. V. Berry,
		\emph{Regular and irregular semiclassical wavefunctions},
		J. Phys. A: Math. Gen. \textbf{10} (1977), no.~12, 2083--2091.
		
		\bibitem{BonthonneauZelditch}
		Y. Bonthonneau and S. Zelditch,
		\emph{Quantum ergodicity for Eisenstein functions},
		C. R. Math. Acad. Sci. Paris \textbf{354} (2016), no.~9, 907--911.
		
		\bibitem{BSS}
		S. B\"ocherer, P. Sarnak, and R. Schulze-Pillot,
		\emph{Arithmetic and equidistribution of measures on the sphere},
		Comm. Math. Phys. \textbf{242} (2003), nos.~1--2, 67--80.
		
		\bibitem{BonifacioHyperbolic}
		J. Bonifacio,
		\emph{Bootstrap bounds on closed hyperbolic manifolds},
		J. High Energy Phys. (2022), no.~2, Article~025.
		
		\bibitem{ButtcaneKhan}
		J. Buttcane and R. Khan,
		\emph{On the fourth moment of Hecke--Maass forms and the random wave conjecture},
		Compos. Math. \textbf{153} (2017), no.~7, 1479--1511.
		
		\bibitem{ColinDeVerdiere}
		Y. Colin de Verdi\`ere,
		\emph{Ergodicit\'e et fonctions propres du laplacien},
		Comm. Math. Phys. \textbf{102} (1985), no.~3, 497--502.
		
		\bibitem{Cuccagna}
		S. Cuccagna,
		\emph{$L^2$ estimates for averaging operators along curves with two-sided $k$-fold singularities},
		Duke Math. J. \textbf{89} (1997), no.~2, 203--216.
		
		\bibitem{DjankovicKhan}
		G. Djankovi\'c and R. Khan,
		\emph{On the random wave conjecture for Eisenstein series},
		Int. Math. Res. Not. IMRN (2020), no.~23, 9694--9716.
		
		\bibitem{DjankovicKhanTruncated}
		G. Djankovi\'c and R. Khan,
		\emph{The fourth moment of truncated Eisenstein series},
		Comm. Math. Phys. \textbf{407} (2026), Article~3.
		
		\bibitem{Dyatlov2023}
		S. Dyatlov,
		\emph{Quantum ergodicity in theorems and pictures},
		Notices Amer. Math. Soc. \textbf{70} (2023), no.~10, 1592--1601.
		
		\bibitem{DyatlovJin}
		S. Dyatlov and L. Jin,
		\emph{Semiclassical measures on hyperbolic surfaces have full support},
		Acta Math. \textbf{220} (2018), no.~2, 297--339.
		
		\bibitem{DyatlovJinNonnenmacher}
		S. Dyatlov, L. Jin, and S. Nonnenmacher,
		\emph{Control of eigenfunctions on surfaces of variable curvature},
		J. Amer. Math. Soc. \textbf{35} (2022), no.~2, 361--465.
		
		\bibitem{GreenleafSeeger2002}
		A. Greenleaf and A. Seeger,
		\emph{Oscillatory integral operators with low-order degeneracies},
		Duke Math. J. \textbf{112} (2002), no.~3, 397--420.
		
		\bibitem{Guo}
		C. Guo,
		\emph{Joint cubic moment of Eisenstein series and Hecke-Maass cusp forms},
		J. Number Theory \textbf{276} (2025), 162--197.
		
		\bibitem{HassellTacy}
		A. Hassell and M. Tacy,
		\emph{Semiclassical $L^p$ estimates of quasimodes on curved hypersurfaces},
		J. Geom. Anal. \textbf{22} (2012), no.~1, 74--89.
		
		\bibitem{HejhalRackner}
		D. A. Hejhal and B. N. Rackner,
		\emph{On the topography of Maass waveforms for $\mathrm{PSL}(2,\mathbb Z)$},
		Experiment. Math. \textbf{1} (1992), no.~4, 275--305.
		
		\bibitem{Huang}
		B. Huang,
		\emph{The cubic moment of Hecke--Maass cusp forms and moments of $L$-functions},
		Math. Ann. \textbf{389} (2024), no.~1, 899--945.
		
		\bibitem{HuangVariance}
		B. Huang,
		\emph{Quantum variance for Eisenstein series},
		Int. Math. Res. Not. IMRN (2021), no.~2, 1224--1248.
		
		\bibitem{HuangLester}
		B. Huang and S. Lester,
		\emph{Quantum variance for dihedral Maass forms},
		Trans. Amer. Math. Soc. \textbf{376} (2023), no.~1, 643--695.
		
		\bibitem{HuangLiCubic}
		B. Huang and L. Li,
		\emph{Quantum variance for cubic moment of Hecke--Maass cusp forms and Eisenstein series},
		arXiv:2510.12322, 2025.
		
		\bibitem{Humphries}
		P. Humphries,
		\emph{Equidistribution in shrinking sets and $L^4$-norm bounds for automorphic forms},
		Math. Ann. \textbf{371} (2018), nos.~3--4, 1497--1543.
		
		\bibitem{HumphriesKhan}
		P. Humphries and R. Khan,
		\emph{On the random wave conjecture for dihedral Maa\ss\ forms},
		Geom. Funct. Anal. \textbf{30} (2020), no.~1, 34--125.
		
		\bibitem{Iwaniec}
		H. Iwaniec,
		\emph{Spectral Methods of Automorphic Forms}, 2nd ed.,
		Graduate Studies in Mathematics, vol.~53,
		American Mathematical Society, Providence, RI, 2002.
		
		
		\bibitem{JakobsonEisenstein}
		D. Jakobson,
		\emph{Quantum unique ergodicity for Eisenstein series on
			$\mathrm{PSL}_2(\mathbb Z)\backslash\mathrm{PSL}_2(\mathbb R)$},
		Ann. Inst. Fourier (Grenoble) \textbf{44} (1994), no.~5, 1477--1504.
		
		\bibitem{Janson}
		S. Janson,
		\emph{Gaussian Hilbert Spaces},
		Cambridge Tracts in Mathematics, vol.~129,
		Cambridge University Press, Cambridge, 1997.
		
		\bibitem{Ki}
		H. Ki,
		\emph{$L^4$-norms and sign changes of Maass forms},
		Ann. of Math. (2) \textbf{204} (2026),
		no.~2, 741--775.
		
		
		\bibitem{Lindenstrauss}
		E. Lindenstrauss,
		\emph{Invariant measures and arithmetic quantum unique ergodicity},
		Ann. of Math. (2) \textbf{163} (2006), no.~1, 165--219.
		
		\bibitem{LuoSarnak1995}
		W. Luo and P. Sarnak,
		\emph{Quantum ergodicity of eigenfunctions on $\mathrm{PSL}_2(\mathbb Z)\backslash\mathbb H^2$},
		Publ. Math. Inst. Hautes \'Etudes Sci. \textbf{81} (1995), 207--237.
		
		\bibitem{LuoSarnakVariance}
		W. Luo and P. Sarnak,
		\emph{Quantum variance for Hecke eigenforms},
		Ann. Sci. \'Ec. Norm. Sup\'er. (4) \textbf{37} (2004), no.~5, 769--799.
		
		\bibitem{MarinucciWigman}
		D. Marinucci and I. Wigman,
		\emph{On nonlinear functionals of random spherical eigenfunctions},
		Comm. Math. Phys. \textbf{327} (2014), no.~3, 849--872.
		
		\bibitem{McDonaldKaufman}
		S. W. McDonald and A. N. Kaufman,
		\emph{Wave chaos in the stadium: Statistical properties of short-wave solutions of the Helmholtz equation},
		Phys. Rev. A \textbf{37} (1988), no.~8, 3067--3086.
		
		\bibitem{MelroseTaylor}
		R. B. Melrose and M. E. Taylor,
		\emph{Near peak scattering and the corrected Kirchhoff approximation for a convex obstacle},
		Adv. Math. \textbf{55} (1985), no.~3, 242--315.
		
		\bibitem{NIST}
		F. W. J. Olver, D. W. Lozier, R. F. Boisvert, and C. W. Clark (eds.),
		\emph{NIST Handbook of Mathematical Functions},
		Cambridge University Press, Cambridge, 2010.
		
		\bibitem{Nonnenmacher}
		S. Nonnenmacher,
		\emph{Anatomy of quantum chaotic eigenstates},
		S\'eminaire Poincar\'e \textbf{14} (2010), 177--220.
		
		\bibitem{Olver}
		F. W. J. Olver,
		\emph{Asymptotics and Special Functions},
		A K Peters, Wellesley, MA, 1997.
		
		\bibitem{PanSogge}
		Y. Pan and C. D. Sogge,
		\emph{Oscillatory integrals associated to folding canonical relations},
		Colloq. Math. \textbf{60/61} (1990), no.~2, 413--419.
		
		\bibitem{Patterson1975}
		S. J. Patterson,
		\emph{The Laplacian operator on a Riemann surface},
		Compositio Math. \textbf{31} (1975), no.~1, 83--107.
		
		\bibitem{Petersen}
		P. Petersen,
		\emph{Riemannian Geometry}, 3rd ed.,
		Graduate Texts in Mathematics, vol.~171,
		Springer, Cham, 2016.
		
		\bibitem{RudnickSarnak}
		Z. Rudnick and P. Sarnak,
		\emph{The behaviour of eigenstates of arithmetic hyperbolic manifolds},
		Comm. Math. Phys. \textbf{161} (1994), no.~1, 195--213.
		
		\bibitem{SarnakProducts}
		P. Sarnak,
		\emph{Integrals of products of eigenfunctions},
		Int. Math. Res. Not. (1994), no.~6, 251--260.
		
		\bibitem{Sarnak2003}
		P. Sarnak,
		\emph{Spectra of hyperbolic surfaces},
		Bull. Amer. Math. Soc. (N.S.) \textbf{40} (2003), no.~4, 441--478.
		
		\bibitem{Sarnak2011}
		P. Sarnak,
		\emph{Recent progress on the quantum unique ergodicity conjecture},
		Bull. Amer. Math. Soc. (N.S.) \textbf{48} (2011), no.~2, 211--228.
		
		\bibitem{SarnakZhao}
		P. Sarnak and P. Zhao, with an appendix by M. Woodbury,
		\emph{The quantum variance of the modular surface},
		Ann. Sci. \'Ec. Norm. Sup\'er. (4) \textbf{52} (2019), no.~5, 1155--1200.
		
		\bibitem{Shnirelman}
		A. I. Shnirel'man,
		\emph{Ergodic properties of eigenfunctions},
		Uspekhi Mat. Nauk \textbf{29} (1974), no.~6 (180), 181--182.
		
		\bibitem{Sogge1988}
		C. D. Sogge,
		\emph{Concerning the $L^p$ norm of spectral clusters for second-order elliptic operators on compact manifolds},
		J. Funct. Anal. \textbf{77} (1988), no.~1, 123--138.
		
		\bibitem{Sogge}
		C. D. Sogge,
		\emph{Fourier Integrals in Classical Analysis}, 2nd ed.,
		Cambridge Tracts in Mathematics, vol.~210,
		Cambridge University Press, Cambridge, 2017.
		
		\bibitem{Soundararajan}
		K. Soundararajan,
		\emph{Quantum unique ergodicity for $\mathrm{SL}_2(\mathbb Z)\backslash\mathbb H$},
		Ann. of Math. (2) \textbf{172} (2010), no.~2, 1529--1538.
		
		\bibitem{Srednicki}
		M. Srednicki,
		\emph{Chaos and quantum thermalization},
		Phys. Rev. E \textbf{50} (1994), no.~2, 888--901.
		
		\bibitem{Watson}
		T. C. Watson,
		\emph{Rankin triple products and quantum chaos},
		Ph.D. thesis, Princeton University, 2002.
		Revised version, 2008, arXiv:0810.0425.
		
		\bibitem{WeichWolf}
		T. Weich and L. L. Wolf,
		\emph{Absence of principal eigenvalues for higher rank locally symmetric spaces},
		Comm. Math. Phys. \textbf{403} (2023), no.~3, 1275--1295.
		
		\bibitem{Wyman}
		E. L. Wyman,
		\emph{Triangles and triple products of Laplace eigenfunctions},
		J. Funct. Anal. \textbf{282} (2022), no.~8, Article~109404.
		
		\bibitem{Zelditch1987}
		S. Zelditch,
		\emph{Uniform distribution of eigenfunctions on compact hyperbolic surfaces},
		Duke Math. J. \textbf{55} (1987), no.~4, 919--941.
		
		\bibitem{Zelditch1991}
		S. Zelditch,
		\emph{Mean Lindel\"of hypothesis and equidistribution of cusp forms
			and Eisenstein series},
		J. Funct. Anal. \textbf{97} (1991), no.~1, 1--49.
		
	\enlargethispage{3\baselineskip}
	\end{thebibliography}
\end{document}